\documentclass[11pt]{amsart}
\usepackage{amsmath,amsthm,amsfonts,amscd,flafter,epsf,amssymb,ifsym,wasysym}
\usepackage{epsfig,graphicx,color}
\usepackage{amsfonts}
\usepackage{amssymb}
\usepackage{amsmath}
\usepackage{amsthm}
\usepackage{latexsym}
\usepackage{amscd}
\usepackage{epsf}
\input{standardcommand}
\begin{document}

\title{On the kappa ring of $\Modbar_{g,n}$}%
\author{Eaman Eftekhary}%
\address{School of Mathematics, Institute for Research in Fundamental Science (IPM),
P. O. Box 19395-5746, Tehran, Iran}%
\email{eaman@ipm.ir}
\author{Iman Setayesh}%
\address{School of Mathematics, Institute for Research in Fundamental Science (IPM),
P. O. Box 19395-5746, Tehran, Iran}%
\email{setayesh@ipm.ir}

\begin{abstract}
Let $\kappa_e(\Mgnbar)$ denote the kappa ring of $\Mgnbar$ in codimension $e$
(equivalently, in degree $d=3g-3+n-e$). For $g,e\geq 0$ fixed, as the number $n$ of the 
markings grows large we show that the 
rank of $\kappa_e(\Mgnbar)$ is asymptotic to
\begin{displaymath}
\frac{{n+e\choose e}{g+e\choose e}}{(e+1)!}\simeq \frac{{g+e\choose e}n^e}{e!(e+1)!}.
\end{displaymath} 
When $g\leq 2$ we show that a kappa class $\kappa\in\kring$ is trivial if and only if 
the integral of $\kappa$ against all  boundary strata is trivial.
For $g=1$ we further show that
the rank of  $\kappa_{n-d}(\Mbar_{1,n})$ 
is equal to $|\PP_1(d,n-d)|$, where 
$\PP_i(d,k)$ denotes the set of partitions $\p=(p_1,...,p_\ell)$ of $d$ 
such that at most $k$ of the numbers $p_1,...,p_\ell$ are greater than $i$.
\end{abstract}
\maketitle
\section{Introduction}\label{sec:intro}
Let $\epsilon : \overline{\mathcal{M}}_{g,n+1} \to 
\overline{\mathcal{M}}_{g,n}$ denote the universal curve over the moduli  space $\Mgnbar$ 
of stable genus $g$,  $n$-pointed curves.  
Let $\mathbb{L}_{i} \to \overline{\mathcal{M}}_{g,n+1}$ 
denote the cotangent line bundle
 over $\Modbar_{g,n+1}$ with fiber over a  point equal to the cotangent line at
 the $i^{th}$ marking. Define
$$\psi_i= c_1(\mathbb{L}_{i})\in A^{1}(\overline{\mathcal{M}}_{g,n+1})\ \ 
\text{and}\ \ \kappa_i = \epsilon_{*}(\psi_{n+1}^{i+1}) 
\in A^{i}(\overline{\mathcal{M}}_{g,n}).$$
The push forwards of the $\kappa$ and $\psi$ classes from the boundary 
strata generate the 
tautological ring $R^*(\overline{\mathcal{M}}_{g,n})$ \cite{F,GP}.
The kappa ring $\kring$  is  the subring of $R^*(\Mgnbar)$ generated by 
$\kappa_1,\kappa_2,...$ over $\Q$. Let $\kappa^d(\Mgnbar)$ denote the 
$\Q$-module generated by the kappa classes of degree $d$ and set 
$\kappa_e(\Mgnbar)=\kappa^{3g-3+n-e}(\Mgnbar)$.\\

Applying the localization formula of \cite{Rahul-localization} to the action of 
$\C^*$ on the moduli space of stable maps from curves of genus $g$ to $\Pp^1$ 
we prove the following theorem.

\begin{theorem}\label{thm:main2}
Fix the genus $g$ and the codimension $e$. As the number $n$ of the marked points 
grows large, the rank of $\kappa_e(\Mgnbar)$ is asymptotic to 
\begin{displaymath}
\frac{{g+e\choose e}{n+e\choose e}}{(e+1)!}\simeq \frac{{g+e\choose e}n^e}{e!(e+1)!}.
\end{displaymath} 
\end{theorem}

Let $G$ be  a connected graph which is decorated by assigning a genus to each one of 
its vertices and let the markings $1,2,...,n$ get distributed among the vertices of $G$.
For a vertex $v\in V(G)$ let $g_v$ denote the genus associated with $v$, $d_v$ denote 
the degree of $v$, and $n_v$ denote the number of markings assigned to $v$.
If $2g_v+n_v+d_v>2$ for every $v\in V(G)$, $G$ is called 
a {\emph{\SCC}}. Every \SCC
$G$ describes a {\emph{combinatorial cycle}} $[G]$ in $\Mbar_{g,n}$ where 
$$g=|E(G)|-|V(G)|+1+\sum_{v\in V(G)}g_v.$$
The tautological ring of $[G]$ is denoted by $R^*([G])$.
An element $\kappa\in\kring$ is called {\emph{combinatorially}} trivial if for all relevant
 stable weighted graphs $G$ as above $\int_{[G]}\kappa=0$. 
 Let $\kappa_0^*(\Mgnbar)\subset \kring$
denote the set of  combinatorially trivial classes and
 $\kappa_c^*(\Mbar_{g,n})$ denote the quotient 
$\kappa^*(\Mbar_{g,n})/\kappa_0^*(\Mgnbar)$, which sits in the short exact sequence 
\begin{displaymath}
0\lra \kappa_0^*\left(\Mgnbar\right) \lra \kappa^*\left(\Mgnbar\right)
\xrightarrow{\ \pi_{g,n}\ }\kappa_c^*\left(\Mgnbar\right) \lra 0.
\end{displaymath}
 A more careful examination of the localization terms  in our argument proves
the following theorem.
\begin{theorem}\label{thm:main1}
The  map $\pi_{g,n}$ is an isomorphism
 of  graded algebras for $g\leq 2$. 
\end{theorem}
 Theorem~\ref{thm:main1} is a consequence of Keel's Theorem~\cite{Keel}
when $g=0$ and follows from Petersen's work \cite{Pet}
on the structure of the tautological ring for $g=1$. Our argument in genus one 
is, however, different from Petersen's argument.\\

Let $\PP(d)$ denote the set of partitions of $d$ and 
$\PP_i(d,k)$ denote the set of  $\p=(p_1,...,p_\ell)\in\PP(d)$ 
such that at most $k$ of the numbers $p_1,...,p_\ell$ are greater than $i$.
Combining Theorem~\ref{thm:main1} with combinatorial arguments, the following
theorem is also proved in this paper.

\begin{theorem}
The rank of  $\kappa^d(\Mbar_{1,n})$ is equal to $|\PP_1(d,n-d)|$.
\end{theorem}

Let us now describe our strategy for bounding the rank of the kappa ring.
Let $\pi_{g,n}^\ell:\Modbar_{g,n+\ell}\ra \Modbar_{g,n}$ 
denote the forgetful map
which forgets the last $\ell$ markings.
For every  $\p=(p_1,...,p_\ell)\in\PP(d)$ let 
\begin{displaymath}
\langle\p\rangle_{g,n}:=\left(\pi_{g,n}^\ell\right)_*\left(
\prod_{i=1}^\ell \frac{1}{1-p_i\psi_{i+n}}\right)\in \kring.
\end{displaymath}
 Let $\langle\p\rangle_{g,n}^j$ denote the degree $j$ part of $\langle \p\rangle_{g,n}$.
The bracket classes $\left\{\langle \p\rangle_{g,n}^d\right\}_{\p\in\PP(d)}$ generate 
the kappa ring $\kappa^d(\Modbar_{g,n})$ 
(Lemma~\ref{lem:basis}, c.f. Proposition 3 from \cite{Rahul-lambdag}).
For  every positive integer $l$, every partition
$$\n=(n_1,..,n_m)\in\PP(2d-2g+2-n-l)$$ and every partition
$\p=(p_1,...,p_\ell)\in\PP(d)$  let 
\begin{displaymath}
C_{\n}^{\p}:=\frac{(-1)^{l+\ell}}{\mathrm{Aut}(\p)}\prod_{i=1}^\ell
\frac{p_i^{p_i-1}}{p_i!}\sum_{\phi}\prod_{i=1}^{m}p_{\phi(i)}^{1-n_i},
\end{displaymath}
where the last sum is over all injections $\phi:\{1,...,m\}\ra \{1,...,\ell\}$. Let 
\begin{displaymath}
 J^{lead}(\n):=\sum_{\p\in\PP(d)}C_{\n}^{\p}\big\langle \p\big\rangle^d_{g,n}\in 
 \kappa^d\left(\Mbar_{g,n}\right).
 \end{displaymath}
 The kappa classes $J^{lead}(\n)$ are trivial over $\Mod_{g,n}^c\subset \Mgnbar$ 
 (the moduli space  of curves of compact type) \cite{Rahul-k}. 
 If $\kappa\in\kring$ has trivial integrals over all 
 combinatorial cycles which only contain genus zero components, then  $\kappa$ 
  is a linear combination of the classes $J^{lead}(\n)$.
In particular, so is every element of $\kappa_0^*(\Mgnbar)$.\\

A \SCC $G$ is called a comb graph if $G$ is a tree,  contains a distinguished vertex 
$v_\infty$, and every vertex $v\in V(G)\setminus\{v_\infty\}$ is connected with
an edge $e_v$ to $v_\infty$. Furthermore, the markings $1,...,n$ are all assigned
to $v_\infty$. If the sum of the genera associated with the vertices of $G$ is $g$,
we get an embedding 
$$\imath^G:\frac{[G]}{\Aut(G)}\lra \Mgnbar,$$
of the quotient of $[G]$ by the group of automorphisms in $\Mgnbar$.
The genus associated with $v_\infty$ is denoted by $g_\infty(G)$.
 Since every $\kappa\in\kappa_0^*(\Mgnbar)$ is a linear combination of the classes 
 $J^{lead}(\n)$ the localization argument of Section 8 from \cite{Rahul-k} gives a 
 presentation 
\begin{displaymath}
\kappa=\sum_{\substack{G:\text{comb}\\ g_\infty(G)<g}}\imath^G_*(\psi_G(\kappa)),
\ \ \ \ \ \psi_G(\kappa)\in R^*([G]).
\end{displaymath}
In particular, for $g=1$, there is only one comb graph $G$ with $g_\infty(G)<1$.
For this comb graph $[G]\simeq \Mbar_{1,1}\times \Mbar_{0,n+1}$.
The above argument implies that every element $\kappa\in\kappa_0(\Mbar_{1,n})$ is 
of the form
$$\kappa=\imath^G_*\big[\pi_1^*(\lambda_1)\pi_2^*(\psi)\big],\ \ \ \text{for some }
\psi\in R^*(\Mbar_{0,n+1}).$$  
If a combinatorially trivial class in the tautological ring of $\Mbar_{1,n}$
takes the form of the right-hand-side of the above equation  
one can quickly conclude that $\psi=0$, and thus $\kappa=0$.\\

\begin{figure}
\begin{center}
\includegraphics[scale=1]{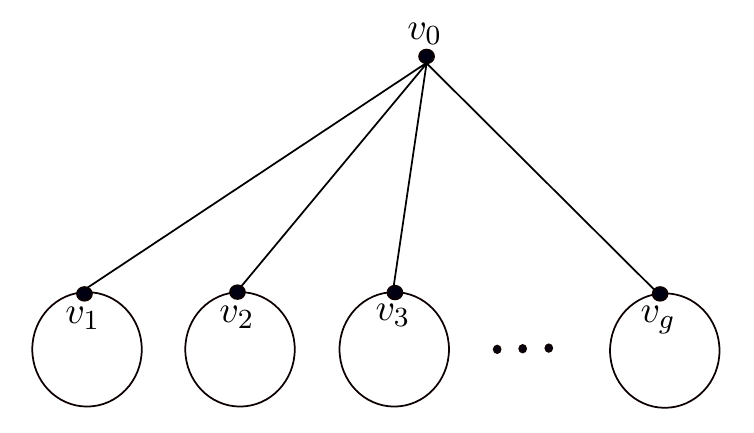}
 \caption{\label{fig:fig1}
{The combinatorial cycle isomorphic to $\Mbar_{0,n+g}$ in $\Mgnbar$.
The genus associated with all vertices is $0$ and all the $n$ markings 
are assigned to $v_0$.}}
\end{center}
\end{figure}

In general, an inductive use of the above procedure, which represents  the elements
of an appropriate subspace  of  $R^*(\Mgnbar)$ in terms of the push-forwards of 
other tautological classes from boundary strata, 
is used in this paper. Let $G_{g,n}$ denote the 
\SCC whose underlying graph is illustrated in Figure~\ref{fig:fig1}. Thus 
$V(G_{g,n})=\{v_0,...,v_g\}$, $G_{g,n}$ has $2g$ edges and   all the markings 
$1,...,n$ are assigned to $v_0$. The \SCC $G_{g,n}$ determines an embedding
\begin{displaymath}
\imath^{g,n}:\frac{\Mbar_{0,n+g}}{S_g}\simeq 
\frac{\Mbar_{0,n+g}\times \overbrace{\Mbar_{0,3}\times
 ... \times \Mbar_{0,3}}^{g\ \text{copies}}}{S_g}=\frac{[G_{g,n}]}{\Aut(G_{g,n})}
 \lra \Mgnbar.
\end{displaymath}
An inductive use of the above reduction scheme  shows that the rank of 
\begin{displaymath}
Q^d(g,n):=\frac{\kappa_0^d\left(\Mgnbar\right)}
{\kappa_0^d\left(\Mgnbar\right)\cap\imath^{g,n}_*
\left(R^{d-2g}\left(\Mbar_{0,n+g}\right)\right)}
\end{displaymath}
is small, compared to the rank of $\kappa_c^d(\Mgnbar)$. When $g$ is small
$Q^d(g,n)$ is trivial. Moreover, for every \SCC 
$G$ which determines a combinatorial cycle in $\Mbar_{0,n+g}$ there is a 
constant $C_G$ and a corresponding  \SCC $\pi(G)$ which determines a combinatorial cycle 
in $\Mgnbar$ such that 
\begin{displaymath}
\int_{[G]}\psi=C_G\int_{[\pi(G)]}\imath^{g,n}_*(\psi).
\end{displaymath}
In particular, by Keel's Theorem \cite{Keel}
\begin{displaymath}
\kappa_0^d\left(\Mgnbar\right)\cap
 \imath^{g,n}_*\left(R^{d-2g}\left(\Mbar_{0,n+g}\right)\right)=0
\end{displaymath}

The above scheme  describes the heart of our argument in this paper.
Motivated by Theorem~\ref{thm:main1} we ask the following question.
\begin{question}\label{question}
Is it true that $\kappa^*_0(\Mgnbar)$ is trivial?
\end{question}

The authors started an investigation of the structure of   
$\kappa_c^*(\Mbar_{g,n})$ in \cite{ES-k1} and proved that its rank  in codimension $e$,
as the number $n$ of the marked points grows large, is asymptotic to 
${{g+e\choose e}{n+e\choose e}}/{(e+1)!}$.
Thus, Theorem~\ref{thm:main2} provides some evidence in support of an 
affirmative answer to Question~\ref{question}, while Theorem~\ref{thm:main1} 
gives an affirmative answer to the aforementioned question for $g\leq 2$. \\

{\bf{Acknowledgement.}} 
The second author was partially supported by a grant from Iran's National Elites 
Foundation.\\

\section{Combinatorial cycles and the $\psi$ classes}\label{sec:background}
It is sometimes more convenient to use alternative bases for the kappa ring of
$\Mgnbar$, instead of the kappa classes. Let
$$\pi_{g,n,k}^{m}:\Modbar_{g,n+m}\ra \Mbar_{g,n+k}$$
denote the forgetful map which forgets the last
$m-k$ markings.
\begin{defn}\label{defn:psi-classes}
For every multi-set $\p=(p_1\geq p_2\geq ...\geq p_m)$ of positive integers
with $p_k>1$ and $p_{i}=1$ for $k<i\leq m$ define
\begin{itemize}
\item $|\p|:=m$ and $d(\p):=\sum_{i=1}^m p_i$
\item $\p^-:=(p_1-1\geq ...\geq p_{k}-1)\in \PP(d(\p)-m)$
\item $\psi(\p):=\psi(p_1,...,p_{m}):=\left(\pi_{g,n,0}^{m}\right)_*
\left(\prod_{i=1}^m\psi_{n+i}^{p_i+1}\right)\in\kappa^{d(\p)}(\Mgnbar)$
\item $\kappa(\p):=\kappa(p_1,...,p_m):=\prod_{i=1}^m\kappa_{p_i}
\in\kappa^{d(\p)}(\Mgnbar)$
\item $\langle \p\rangle_{g,n;k}:
=\left(\pi_{g,n,k}^{m}\right)_*\left(
\frac{1}{m!}\sum_{\sig\in S_m}\prod_{i=1}^m
\frac{1}{1-p_{\sig(i)}\psi_{n+i}}\right)\in R^*(\Mbar_{g,n+k})$.
\end{itemize}
Let $\langle\p\rangle=\langle\p\rangle_{g,n;0}$ and  let $\langle \p\rangle^j$ 
denote the degree $j$ part of $\langle\p\rangle$. Similarly, let $\langle \p\rangle_{g,n;k}^j$ 
denote the degree $j$ part of $\langle\p\rangle_{g,n;k}$ and set
$$\langle \p\rangle^{g,n;k}_j:=\langle \p\rangle_{g,n;k}^{3g-3+n+k-j}.$$ 
\end{defn}

\begin{lem}\label{lem:basis}
The subsets of $\Acal^d(\Mgnbar)$ defined by
$$\Big\{ \psi(\p)\ \big|\ \p\in\PP(d)\Big\},\ \ 
\Big\{ \kappa(\p)\ \big|\ \p\in\PP(d)\Big\}\ \ \text{and} \ \
\Big\{ \langle\p\rangle^d \big|\ \p\in\PP(d)\Big\}
$$
are related by invertible linear transformations.
\end{lem}
\begin{proof}
The transformation relating the first two sets  
(which is  independent of $g$ and $n$) is due to Faber and is 
discussed in \cite{AC}. The transformation relating the 
first set to the third set is discussed in 
Proposition 3 from \cite{Rahul-lambdag}. This later transformation only depends
on $2g-2+n$.
\end{proof}

Let $\Psi(d)$ denote the formal vector space over $\Q$ which is freely generated by 
the partitions $\p\in \PP(d)$. 
There are surjections 
\begin{displaymath}
\psi_{g,n},\kappa_{g,n},\langle\rangle_{g,n}:\Psi(d)\lra \kappa^d(\Mgnbar)
\end{displaymath}
which are defined by
\begin{displaymath}
\begin{split}
\psi_{g,n}(\sum_{\p\in\PP(d)}a_\p.\p):=\sum_{\p\in\PP(d)}a_\p \psi(\p),\ \ \ \ \
&\kappa_{g,n}(\sum_{\p\in\PP(d)}a_\p.\p):=\sum_{\p\in\PP(d)}a_\p \kappa(\p)\\
\text{and}\ \ \ \ \Big\langle\sum_{\p\in\PP(d)}a_\p.\p\Big\rangle_{g,n}
&:=\sum_{\p\in\PP(d)}a_\p \langle\p\rangle
\end{split}
\end{displaymath}
Lemma~\ref{lem:basis} implies that there are invertible matrices $P_d:\Psi(d)\ra \Psi(d)$
and $Q_{d,m}:\Psi(d)\ra \Psi(d)$ for $m\in \Z^{+}$ such that
\begin{displaymath}
\psi_{g,n}=\kappa_{g,n}\circ P_d\ \ \ \text{and}\ \ \ 
\langle\rangle_{g,n}=\kappa_{g,n}\circ Q_{d,n+2g}.
\end{displaymath}

In \cite{Rahul-k} Pandharipande shows that associated with every pair of partitions 
$$\mm\in\PP(d)\setminus\PP(d,2g-2+n-d)\ \ \ \text{and}\ \ \ 
\p\in\PP(d)$$
there is a  rational number $C_{\mm^-}^\p$ with the property that 
\begin{itemize}
\item The matrix $\left(C_{\mm^-}^\p\right)_{\mm,\p}$ is of full-rank, 
with rank equal to 
$$|\PP(d)|-|\PP(d,2g-2+n-d)|.$$
\item If we set 
$$j(\mm^-):=\sum_{\p\in\PP(d)}C_{\mm^-}^\p.\p\in\Psi(d)$$
 the restriction of $\langle j(\mm^-)\rangle_{g,n}\in\kring$ to $\Mod_{g,n}^c$
is trivial.\\
\end{itemize}

For $g=0$, these relations among the kappa classes are all the possible relations. 
In other words, the kernel of the map 
\begin{displaymath}
\langle\rangle_{0,n}:\Psi(d)\lra \kappa^d(\Mbar_{0,n})
\end{displaymath}
is generated by $\{j(\m^-)\}_{\m}$. This observation gives a surjection
$$T^c_{g,n}:\kappa(\Mbar_{0,n+2g})\ra \kappa(\Mod_{g,n}^c).$$
This map {\emph{translates}} every kappa class over $\Mbar_{0,n+2g}$ to a kappa 
class over $\Mod_{g,n}^c$, which comes from the same formal expression. Nevertheless,
the aforementioned map does not have a clear geometric meaning 
(at least to the authors). We will encounter such 
homomorphisms again when we try to obtain relations among the kappa classes
in this paper.  \\

\begin{defn}
A {\emph{weighted graph}} $G$ is a finite connected 
graph with the set $V(G)$ of vertices, 
the set $E(G)$ of edges and a weight function
$$\epsilon=\epsilon_G:V(G)\ra \Z^{\geq 0}\times 2^{\{1,...,n\}},$$
where $2^{\{1,...,n\}}$ denotes the set of subsets of $\{1,...,n\}$.
For $i\in V(G)$ denote the degree of $i$ by $d_i=d(i)$ and let $\epsilon(i)=(g_i,I_i)$.
$G$ is called a {\emph{\SCC}} if $\{I_i\}_{i\in V(G)}$ is a partition of $\{1,...,n\}$ and 
 for every vertex $i\in V(G)$, $2g_i+|I_i|+d_i>2$. Define $n(G)=n$ and 
$$g(G):=\Big(\sum_{i\in V(G)}g_i\Big)+|E(G)|-|V(G)|+1.$$
\end{defn}

Associated with a \SCC $G$ there is a natural map
$$\imath_G:\Ccal(G)=\prod_{i\in V(G)}\Modbar_{g_i,|I_i|+d_i}\lra \Mbar_{g(G),n(G)}$$
which is an embedding after we mod out the source by its  automorphisms.
Thus, a \SCC $G$ determines a {\emph{combinatorial cycle}}
$$[G]:=(\imath_G)_*\left[\Ccal(G)\right]\in A_{d}\left(\Modbar_{g(G),n(G)}\right),$$
where $d=3g(G)-3+n(G)-|E(G)|$. 
\\

Let $H=H_{g,n}$ be a \SCC with a single vertex $v$, $g$ self edges from $v$ to itself, and 
with $\epsilon(v)=(0,\{1,...,n\})$. $H$ determines a homomorphism 
$\imath^H:\Mbar_{0,n+2g}\ra \Mgnbar$. If $G$ is a \SCC with 
$g(G)=0$ and $n(G)=n+2g$ then
\begin{displaymath}
\int_{\imath^H_*[G]} \langle\q\rangle_{g,n}=\frac{1}{24^g\times g!}
\int_{[G]} \langle \q\rangle_{0,n+2g}\ \ \ \forall\ \q\in\PP(d).
\end{displaymath}
\begin{remark}\label{remark:rahuls-result}
Let $\kappa=\langle\psi\rangle_{g,n}$ for some $\psi\in\Psi(d)$ is such that 
$\int_{\imath^H_*[G]}\kappa=0$ for all \SCC $G$ with $g(G)=0$ and $n(G)=n+2g$.
By the above observation, $\langle \psi\rangle_{0,n+2g}$ has trivial integral over all 
combinatorial cycles, and is thus trivial by Keel's Theorem \cite{Keel}.
By Pandharipande's result this implies that $\psi=\sum_{\m}a_\m j(\m^{-})$, and 
thus
$$\kappa=\sum_{\m}a_\m \langle j(\m^-)\rangle_{g,n}.$$
\end{remark}

Fix the \SCC $G$ and $\psi(\p)=\psi(p_1,...,p_k)$   with $n=n(G)$, $g=g(G)$,
$d(\p)+|E(G)|=3g-3+n$   and  $V(G)=\{1,...,m\}$. Let 
  $$Q=\left\{(h,r)\in\Z^{\geq 0}\times \Z^{\geq 0}\ \big|\ \ 2h+r>2\right\}.$$
  The {\emph{modified weight}} multi-set associated with
$G$ is the multi-set
\begin{displaymath}
\begin{split}
&\q_G:=\big(\theta_G(i)\in Q\ \big|\ i\in\{1,...,m\}\big),\ \ \ \text{where}\\
&\theta_G(i):=(g_i,m_i=|I_i|+d_i),\ \ \ \forall\ 1\leq i\leq m.\\
\end{split}
\end{displaymath}
The integral
\begin{displaymath}
\big\langle\psi(\p)\ ,\ [G]\big\rangle=\int_{[G]}\psi(\p)
=\int_{(\pi^{m}_{g,n,k})^*[G]}\prod_{j=1}^k \psi_{n+j}^{p_j+1}\in \Q
\end{displaymath}
only depends on the multi set $\q_G$ \cite{ES-k1}. 
We  denote the value of the above integral by $\langle \psi(\p),\q_G\rangle_{g,n}$,
or just $\langle \psi(\p),\q_G\rangle$ if there is no confusion.
We denote by $\QQ(d;g,n)$ the set of all multi-sets
$\q=(\theta_i)_{i=1}^m$  such that $\q=\q_G$ for some 
\SCC $G$ with $g=g(G)$, $n=n(G)$ and $d=3g-3+n-|E(G)|$. 
A kappa class $\kappa\in\kappa^d(\Mgnbar)$ is 
{{combinatorially trivial}} if $\langle\kappa,\q\rangle=0$ for all 
$\q\in\QQ(d;g,n)$. \\

For a partition $\n=(n_1,...,n_k)\in \PP(n)$ of length $k=|\n|$,  let 
\begin{displaymath}
\begin{split}
&\q_0(\n)=\{(0,n_1+2),...,(0,n_k+2)\}\in \QQ(n-k;1,n),\ \ \text{and}\\
&\q_l(\n)=\{(1,l),(0,n_1+2),...,(0,n_k+2)\}\in \QQ(n+l-k;1,n+l),\ \ \ l\geq 1.
\end{split}
\end{displaymath}
Every element of $\QQ(d;1,n)$ is of the  form $\q_l(\n)$ for some 
non-negative integer $l$ and some $\n\in \PP(n-l;n-d)$. Here
 $\PP(d;k)$ denotes the set of the partitions of $d$ into precisely $k$ parts.\\
 
If $\n,\mm\in \PP(d)$ define $\n<\mm$ if $\n$ refines $\mm$. This partial ordering 
may be extended to a total ordering on  $\PP(d)$. We fix one 
such total ordering and will refer to it as the {\emph{refinement ordering}}. 
Define
$$\ppp:\QQ(d;g,n)\lra \PP(d,3g-2+n-d)$$
by sending $\q=\{(g_i,n_i)\}_{i=1}^k$ to $\{3g_i-3+n_i\}_{i=1}^k$.
If $\langle \psi(\n),\q\rangle$ is non-trivial for some $\q\in\QQ(d;g,n)$ then 
$\n$ refines $\ppp(\q)$ \cite{ES-k1}.\\

\section{Localization and moduli space of stable maps to $\Pp^1$}\label{sec:localization}
\subsection{The vanishing cycles}
Fix the integers $m\geq k\geq 0$ and
let $$\Modbar_{g,n+m}(\Pp^1,d)$$ denote the moduli space of stable maps of degree 
$d$ from curves of genus $g$ with $n+m$ marked points to $\Pp^1$ and denote by 
$$\epsilon: \Modbar_{g,n+m}(\Pp^1,d)\lra \Mbar_{g,n+k}$$
 the homomorphism which forgets the map and the last $m-k$ markings.\\

 Let $\C^*$ act  on $V=\C\oplus \C$ by
$$\zeta. (z_1,z_2)=(z_1,\zeta z_2)\ \ \forall\ \zeta\in\C^*,\ (z_1,z_2)\in\C\oplus \C.$$
Let $\pp_0=[0:1]$ and $\pp_\infty=[1:0]$ denote the fixed points of the 
corresponding action on $\Pp^1=\Pp(V)$. 
For every line bundle $L\ra \Pp(V)$, an equivariant lifting of the $\C^*$ action 
to $L$ is determined by the weights $l_0$ and $l_\infty$ of the fiber representations 
$L_0=L|_{\pp_0}$ and $L_{\infty}=L|_{\pp_\infty}$.
The canonical lift of the action  for $T_{\Pp^1}$ has weights $[l_0,l_\infty]=[1,-1]$.\\

Let $\pi:\Ucal_{g,n+m}(\Pp(V),d)\ra \Modbar_{g,n+m}(\Pp(V),d)$ denote the universal curve
and $\mu:\Ucal_{g,n+m}(\Pp(V),d)\ra\Pp(V)$
denote  the universal map. The action of $\C^*$ on $\Pp(V)$ induces $\C^*$ actions 
on $\Ucal_{g,n+m}(\Pp(V),d)$ and $\Modbar_{g,n+m}(\Pp(V),d)$ 
compatible with $\pi$ and $\mu$. Let 
$$\left[\Modbar_{g,n+m}(\Pp(V),d)\right]^{vir}\in A_{2g+2d-2+n+m}^{\C^*}
\left(\Modbar_{g,n+m}(\Pp(V),d)\right)$$
denote the $\C^*$-equivariant virtual fundamental class of $\Modbar_{g,n+m}(\Pp(V),d)$
\cite{Rahul-localization}.\\

Following \cite{Rahul-k}  we consider three types of equivariant chow classes 
over the moduli space  $\Modbar_{g,n+m}(\Pp(V),d)$:
\begin{itemize}
\item The linearization $[0,1]$ on $\Ocal_{\Pp(V)}(-1)$ defines the $\C^*$ action 
on the rank $d+g-1$ bundle
$$\R=R^1\pi_*\left(\mu^*\Ocal_{\Pp(V)}(-1)\right)\lra \Modbar_{g,n+m}(\Pp(V),d).$$
We denote the top Chern class of this bundle by
$$c_{top}(\R)\in A_{\C^*}^{d+g-1}\left(\Modbar_{g,n+m}(\Pp(V),d)\right).$$

\item For each marking $i$, let $\psi_i\in A_{\C^*}^1\left(\Modbar_{g,n+m}(\Pp(V),d)\right)$
denote the first Chern class of the canonically linearized cotangant line corresponding to 
the $i^{th}$ marking. \\

\item With $\mathrm{ev}_i:\Modbar_{g,n+m}(\Pp(V),d)\ra \Pp(V)$ denoting the $i$-th
 evaluation map and with the $\C^*$-linearization $[1,0]$ on $\Ocal_{\Pp(V)}(1)$, let
 $$\rho_i=c_1\left(\mathrm{ev}_i^*\Ocal_{\Pp(V)}(1)\right)\in 
 A_{\C^*}^1\left(\Modbar_{g,n+m}(\Pp(V),d)\right),$$
 while with the $\C^*$ linearization $[0,-1]$ on $\Ocal_{\Pp(V)}(1)$ we let 
 $$\tilde \rho_i=c_1\left(\mathrm{ev}_i^*\Ocal_{\Pp(V)}(1)\right)\in 
 A_{\C^*}^1\left(\Modbar_{g,n+m}(\Pp(V),d)\right).$$
\end{itemize}
Note that in the non-equivariant limit $\rho_i^2=0$, and that 
$\epsilon$ is equivariant with respect to the trivial action on $\Mbar_{g,n+k}$.\\

Fix the cycle dimension $e$ and the sequence $\n=(n_{1},...,n_{m})$ with 
$$\sum_{i=1}^m n_i=d+g-1-e-l,\ \ \  l>0,$$
which determines a partition in $\PP(d+g-1-e-l;m)$, 
 denoted by $\n$ by slight abuse of the notation.
 The partition $\n$ is of the form $\n=\mm^-$,
where $$\mm\in \PP(d)\setminus\PP(d,\max\{e-g+1,k-1\}).$$

 Let $I(\n)=I(\n;d,g,n,k)$ denote the $\C^*$-equivariant 
push-forward
\begin{displaymath}
\epsilon_*\left(\rho_{n+1}^{l}\prod_{i=1}^{m}\rho_{i+n}\psi_{i+n}^{n_i}
\prod_{j=1}^n\tilde \rho_j \ c_{top}(\R)\cap 
\left[\Modbar_{g,n+m}(\Pp(V),d)\right]^{vir}\right).
\end{displaymath}
Since the degree of 
\begin{displaymath}
\rho_{n+1}^{l}\left(\prod_{i=1}^{m}\rho_{i+n}\psi_{i+n}^{n_i}\right)\left(
\prod_{j=1}^n\tilde \rho_j\right) c_{top}(\R)
\end{displaymath}
is  $2d+2g-e-2+n+m$ and the cycle dimension of the virtual fundamental 
class is $2d+2g-2+n+m$, the 
cycle dimension of the class $I(\n)$ is 
$$e=(2d+2g-2+n+m)-(2d+2g-e-2+n+m).$$
In other words, $I(\n)\in A_{\C^*}^{3g-3+n+k-e}\left(\Modbar_{g,n+k}\right)$. 
Since the exponent of
$\rho_{n+1}$ is at least $2$, $I(\n)$ vanishes in the non-equivariant limit.\\

\subsection{The localization terms}
The virtual localization formula of \cite{Rahul-localization} may be used to calculate 
$I(\n)$ in terms of the tautological classes on $\Modbar_{g,n+k}$.
The sum in the localization formula is over connected decorated graphs $\Gamma$ 
(indexing the $\C^*$-fixed loci of $\Modbar_{g,n+m}(\Pp(V),d)$). 
Every vertex of $\Gamma$
either lies over $\pp_0$ or over $\pp_\infty$, 
and is labelled by a genus.  The edges of the graph lie 
over $\Pp^1$ and are labelled with degrees (of the maps corresponding to the edges).
The total sum of these degrees is  equal to $d$. The graphs carry $n+m$ markings over 
their vertices. For a vertex $v$ of $\Gamma$ 
let $d(v)$ denote the degree of $v$.\\

If a graph $\Gamma$ has a vertex $v$ over $\pp_0$ with $d(v)>1$, $v$ yields a trivial 
Chern root of the bundle $\R$ with trivial weight $0$ in the numerator of the localization 
formula, by  our choice of linearization on the bundle $\R$. Hence the contribution of 
such graphs to the sum in the localization formula is trivial.
Thus, only comb graphs $\Gamma$ contribute to $I(\n)$. Every comb graph contains 
a set $V_0=V_0(\Gamma)$ of  vertices  which lie over
 $\pp_0$, and each $v\in V_0$ is connected by an 
edge to a unique vertex $v_\infty$ which lies over $\pp_\infty$. 
The linearization of the classes $\rho_{n+1},...,\rho_{n+m}$ and 
$\tilde\rho_1,...,\tilde\rho_n$ implies that the first $n$ markings lie 
on $v_\infty$ and the last $m$ markings are placed on the vertices in $V_0$.
For every $v\in V_0$ let $g_v$ denote the genus associated with  $v$ and 
let $I_v\subset \{1,...,m\}$ determine the subset of the last $m$ markings 
which is associated with  $v$. Note that the genus associated with  $v_\infty$ is 
 $g_\infty=g-\sum_{v\in V_0}g_v$. Denote the degree associated with 
the edge connecting $v$ to $v_\infty$ by $p_v$. 
The fixed locus associated with the decorated comb graph $\Gamma$ is thus determined 
by a multi-set $\{(g_v,p_v, I_v)\}_{v\in V_0}$ such that $\sum g_v\leq g$,
$\sum_{v}p_v=d$ and $\{I_v\}_v$ is a partition of $\{1,...,m\}$. 
We  abuse the notation and use $\Gamma$ to refer to this associated multi-set. 
The partition $(p_v)_{v\in V_0(\Gamma)}$ of $d$ is denoted by $\p_\Gamma$. 
\\

The group $S_{\Gamma}$ of  permutations $\sigma:V_0\ra V_0$
of the vertices in $V_0=V_0(\Gamma)$ acts on the multi-set associated with $\Gamma$ 
by sending $\{(g_v,p_v,I_v)\}_{v\in V_0}$ to  $\{(g_v,p_{\sig(v)},I_v)\}_{v\in V_0}$.
We denote the image of $\Gamma$ under the action of $\sig\in S_\Gamma$ by 
$\sig(\Gamma)$.
The automorphism group of $\Gamma$ consists of the permutations  
$\sig$ of the vertices such that for every vertex $v\in V_0$ either 
$I_v=I_{\sig(v)}=\emptyset$ and  $g_{\sig(v)}=g_v$, or $\sig(v)=v$. 
We denote the group of automorphisms of 
$\Gamma$ by $\Aut(\Gamma)$.\\

If  $I_v=\{i_1,...,i_{k_v}\}$    the fixed locus corresponding 
to $\Gamma$  contains a product factor 
$\Mbar^{v,\Gamma}\simeq\Modbar_{g_v,k_v+1}$, 
provided that $2g_v+k_v>1$. 
The subset $I_v$ labels $k_v$ of the markings on $\Mbar^{v,\Gamma}$ and 
we use the vertex $v$ itself to label the last marking on this moduli space.
The classes 
$\psi_{i_j+n}^{n_{i_j}}$ carry trivial $\C^*$ weight. Moreover, the integrand term 
$c_{top}(\R)$ yields a factor $\lambda_{g_v}$ on $\Modbar^{v,\Gamma}$. 
Thus, we obtain the class $\lambda_{g_v}\Ps_{v,\Gamma}(\n)\in A^*(\Mbar^{v,\Gamma})$
where
\begin{displaymath}
\psi_{v,\Gamma}(\n):=\prod_{i=1}^{k_v}\psi_{i_j+n}^{n_{i_j}}\in
 A^{n_{i_1}+...+n_{i_{k_v}}}\left(\Modbar^{v,\Gamma}\right)
\end{displaymath}
over this product factor, which is trivial unless 
$$\sum_{j=1}^{k_v}(n_{i_j}-1)\leq 2g_v-2.$$ 
In particular, $(g_v,k_v)\neq (0,i)$ with $i>1$. In other words, if $g_v=0$ the vertex 
$v$ can accommodate at most one of the markings from $\{n+1,...,n+m\}$.\\

Let 
$\Mbar^{\infty,\Gamma}:=\Mbar_{g_\infty, n+|V_0(\Gamma)|}$,
where the last $|V_0(\Gamma)|$ markings are again labelled by the vertices 
in $V_0(\Gamma)$. Denote the subset of genus zero vertices in $V_0$ with no 
markings on them by $V^0=V^0(\Gamma)$, the subset of genus zero vertices $v$ 
with one marking by $V^1=V^1(\Gamma)$ and set 
$V^2=V^2(\Gamma)=V_0\setminus (V^0\cup V^1)$. For $v\in V^1$, if $I_v$ 
consists of the single element $i\in\{1,...,m\}$ we set $n_v=i$.\\

The contribution of the fixed locus corresponding to $\Gamma$ to $I(\n)$ 
may be computed following \cite{Rahul-k}. The only difference is that in this 
case 
\begin{itemize}
\item The contribution from the deformation of the source (i.e. smoothing the nodes)
 adds an extra product factor
\begin{displaymath}
\prod_{v\in V^2(\Gamma)}\frac{1}{\left(\frac{t}{p_v}\right)+\psi_v},
\end{displaymath} 
A power $\psi_v^{m_v}$ of $\psi_v$ thus appears 
in the contribution of $\Gamma$ to $I(\n)$ for smoothing the node corresponding to 
the vertex $v$.
\item The deformation of the map contributes a factor of 
$e(\mathbb{E}^*\otimes \mathbf{1})$ over each one of the components 
in the fixed locus which are mapped to $\pp_0$ (i.e. over each $\Mbar^{v,\Gamma}$ 
with $v\in V^2(\Gamma)$).
The  Euler class of $\mathbb{E}^*\otimes \mathbf{1}$ over 
the product factor corresponding to a vertex  $v\in V^2(\Gamma)$ can contribute 
via a lambda-class $(-1)^{g_v-h_v}\lambda_{h_v}$ for some integer
$0\leq h_v\leq g_v$.  Since $\lambda_{g_v}^2=0$ over $\Mbar^{v,\Gamma}$ we may further
assume that $h_v<g_v$.
\end{itemize}

 The terms corresponding to $\Gamma$ in $I(\n)$ are thus 
 indexed by the set $c(\Gamma)$ of the multi-sets $c=(h_v,m_v)_{v\in V^2(\Gamma)}$ 
with such that 
\begin{itemize}
\item $0\leq h_v<g_v$.
\item $0\leq m_v\leq 2g_v-h_v-2-\sum_{i\in I_v}n_{i}$.
\end{itemize}

Let $\Mbar^\Gamma$ denote the 
fixed locus corresponding to $\Gamma$ and 
$\pi_v:\Mbar^\Gamma\ra \Mbar^{v,\Gamma}$ denote the projection map
over the product factor corresponding to the vertex $v$.
Denote the projection from $\Mbar^\Gamma$ to $\Mbar^{\infty,\Gamma}$ by 
$\pi_\infty$. The restriction of $\epsilon$ to $\Mbar^\Gamma$ gives a map 
from $\Mbar^\Gamma$ to $\Mbar_{g,n+k}$.
The  contribution  corresponding to  
$\Gamma$ and $c=(h_v,m_v)_{v\in V^2(\Gamma)}\in c(\Gamma)$ takes the form
\begin{displaymath}
I(\n,\Gamma,c)=B(\n,\Gamma,c)\
\epsilon_*\left(\Ps(\n,\Gamma,c)\cap \left[\Mbar_{\Gamma}\right]^{vir}\right)
\end{displaymath}
where
\begin{displaymath}
\begin{split}
&\Ps(\n,\Gamma,c):=\left(\prod_{v\in V^2(\Gamma)}\pi_v^*\left(\lambda_{g_v}
\lambda_{h_v}\psi_v^{m_v}\Ps_{v,\Gamma}(\n)\right)\right)
\pi_\infty^*\left(\prod_{v\in V_0(\Gamma)}
\frac{1}{1-p_v\psi_v}\right)_{e(\n,\Gamma,c)}\\
&e(\n,\Gamma,c)=e-\sum_{v\in V^2(\Gamma)}\Big(|I_v|+2g_v-2-h_v-m_v-
\deg(\Ps_{v,\Gamma}(\n))\Big)
\end{split}
\end{displaymath}
and the coefficient $B(\n,\Gamma,c)$ is defined by 
\begin{displaymath}
\frac{1}{\Aut(\p_\Gamma)}
\left(\prod_{v\in V^0(\Gamma)}\frac{p_v^{p_v-1}}{p_v!}\right)\left(
\prod_{v\in V^1(\Gamma)}\frac{p_v^{p_v-m_{v}}}{p_v!}\right)\left(
\prod_{v\in V^2(\Gamma)}\frac{p_v^{p_v+m_v+1}}{p_v!}\right)
\end{displaymath}
Here, for a tautological class $\psi$, $(\psi)_l$ denotes the  part of 
$\psi$ corresponding to the cycle dimension  $l$.\\

For $\Gamma$ as above and $v\in V^2(\Gamma)$, set $J_v=I_v\cap \{1,...,k\}$ and
define 
$$k(\Gamma)=\Big|\left\{n_v\ \big|\ v\in  V^1(\Gamma)\right\}\cap\{1,...,k\}
\Big|+\left|V^2(\Gamma)\right|.$$ 
Correspondingly, set
\begin{displaymath}
\begin{split}
&\Nbar^{v,\Gamma}=\Modbar_{g_v,1+|J_v|},\ \ \forall v\in V^2(\Gamma),\ \ \  \ \ 
\Nbar^{0,\Gamma}=\prod_{v\in V^2(\Gamma)}\Nbar^{v,\Gamma}\\
&\Nbar^{\infty,\Gamma}=\Mbar_{g_\infty,n+k(\Gamma)}
\ \ \ \ \ \ \ \text{and } \ \ \ \ \ \ \ \ 
\Nbar^\Gamma=\Nbar^{0,\Gamma}\times \Nbar^{\infty,\Gamma}.
\end{split}
\end{displaymath}
Let  $\pi^{\infty,\Gamma}:\Mbar^{\infty,\Gamma}\ra \Nbar^{\infty,\Gamma}$ denote
 the forgetful map which forgets the markings $n+k+1,...,n+m$.
 Denote the projection maps from $\Nbar^\Gamma$ to $\Nbar^{0,\Gamma}$ and 
 $\Nbar^{\infty,\Gamma}$ by $q_0$ and $q_\infty$, respectively, and define
 $q_1^*=q_\infty^*\circ \pi_*^{\infty,\Gamma}$. 
The map $\epsilon:\Mbar^\Gamma\ra \Mbar_{g,n+k}$ factors through an embedding 
of $\Nbar^\Gamma$ in $\Mbar_{g,n+k}$. Thus, there are Chow classes 
$\eta_i(\n,\Gamma,c)\in R_{i}(\Nbar^{0,\Gamma})$ with the property that
\begin{equation}\label{eq:reduction}
I(\n,\Gamma,c)=\jmath^\Gamma_*\left(\sum_{i=0}^e
q_0^*\Big(\eta_i(\n,\Gamma,c)\Big)
q_1^*\left(\prod_{v\in V_0(\Gamma)}
\frac{1}{1-p_v\psi_v}\right)_{e-i}\right),
\end{equation}
where $\jmath^{\Gamma}:\Nbar^\Gamma\ra \Mbar_{g,n+k}$ is the embedding 
of the quotient of $\Nbar^\Gamma$ by its automorphisms in $\Mbar_{g,n+k}$.\\
 
Associated with the sequence $\n$, we  obtain the following relation
in the tautological ring of $\Mbar_{g,n+m}$
\begin{equation}\label{eq:relation-general}
\sum_\Gamma \sum_{c\in c(\Gamma)}I(\n,\Gamma,c)=0
\end{equation}
where $I(\n,\Gamma,c)$ is given as in (\ref{eq:reduction}).\\

\subsection{The leading term}
Among all  $\Gamma$, we distinguish the decorated comb graphs 
with $V^2(\Gamma)=\emptyset$, and denote the set of such 
graphs by $\Jcal= \Jcal(g,n,m,d)$. For $\Gamma\in\Jcal$ the set $c(\Gamma)$ 
is trivial, and the corresponding coefficient is 
\begin{displaymath}
B(\n,\Gamma)=\frac{1}{\Aut(\p_\Gamma)}\left(\prod_{v\in V^0(\Gamma)}
\frac{p_v^{p_v-1}}{p_v!}\right)\left(
\prod_{v\in V^1(\Gamma)}\frac{p_v^{p_v-n_{v}}}{p_v!}\right)
\end{displaymath}
We may thus compute 
\begin{displaymath}
I(\n,\Gamma)=B(\n,\Gamma)\ \jmath^\Gamma_*\left(\prod_{v\in V_0(\Gamma)}
\frac{1}{1-p_v\psi_v}\right)_{e}
\end{displaymath}
We call the expression 
\begin{displaymath}
I^{lead}(\n)=\sum_{\Gamma\in\Jcal}I(\n,\Gamma)
\end{displaymath}
the leading term in the relation of Equation (\ref{eq:relation-general}).\\

Associated with a partition in $\PP(d+g-1-e-l;m)$ which is represented by the sequence 
$\n$ there are $m!/\Aut(\n)$  different sequences which may be constructed from $\n$. 
For $\sig\in S_m$, let  $\sig(\n)$ denote the sequence obtained by applying the 
permutation $\sig$ to $\n$. Let 
\begin{displaymath}
\begin{split}
&J(\n,\Gamma,c):=\frac{1}{m!}\sum_{\sig\in S_m}I(\sig(\n),\Gamma,c)\ \ \text{and}\\ 
&J(\n):=\frac{1}{m!}\sum_{\sig\in S_m}I(\sig(\n))=\sum_{\Gamma}\sum_{c\in c(\Gamma)}
J(\n,\Gamma,c).
\end{split}
\end{displaymath}
Note that $J(\n,\Gamma,c)$ and $J(\n)$ depend on the partition associated with $\n$, and 
not the sequence $\n$ itself. \\

Let $\Psi_e(d;g,n,k)$ denote the $\Q$-module generated by the subset
$$\left\{\langle \p\rangle^{g,n;k}_e\ \big|\ \p\in\PP(d)\setminus \PP(d,k-1)\right\}
\subset R_e(\Mbar_{g,n+k}),$$ 
and set $\Psi_e(d;g,n)=\Psi_e(d;g,n,0)$.  
The expressions 
$$J^{lead}(\n)=\frac{1}{m!}\sum_{\sig\in S_m}I^{lead}(\sig(n))$$
belong to $\Psi_e(d;g,n,k)$.\\

Suppose that $\Gamma$ is a decorated comb graph which contributes to $I^{lead}(\n)$.
Associated with $\Gamma$ is the partition $\p_\Gamma=(p_v)_{v\in V_0(\Gamma)}$
in $\PP(d)$. The partition associated with $\sig(\Gamma)$ is $\p_\Gamma$ as well.
Given $\p=(p_1,...,p_\ell)\in\PP(d)$, 
in order to determine the decorated comb graph  
$\Gamma$ which contributes to $I^{lead}(\n)$ and satisfies $\p_\Gamma=\p$
we need to specify an injection $\phi:\{1,...,m\}\ra \{1,...,\ell\}$, which
 determines the $m$ vertices in $V_0(\Gamma)=\{1,...,\ell\}$ which carry 
 the markings $\{n+1,...,n+m\}$. Let 
 $$J(\n,\p)=\sum_{\Gamma\in \Jcal:\ \p_\Gamma=\p} J(\n,\Gamma).$$
 The above observation implies that
 \begin{displaymath}
 \begin{split}
 J(\n,\p)&=\prod_{i=1}^\ell \frac{p_i^{p_i-1}}{p_i!}\sum_{\phi}\prod_{j=1}^m
 p_{\phi(i)}^{1-n_i}\left(\pi_{g,n,k}^{\ell}\right)_*\left(
\frac{1}{\ell !}\sum_{\sig\in S_\ell}\prod_{i=1}^\ell\frac{1}{1-p_{\sig(i)}\psi_{n+i}} 
 \right)_e\\
 &=\prod_{i=1}^\ell \frac{p_i^{p_i-1}}{p_i!}\sum_{\phi}\prod_{j=1}^m
 p_{\phi(i)}^{1-n_i}\Big\langle\p\Big\rangle^{g,n;k}_e
 =C_\n^\p\Big\langle\p\Big\rangle^{g,n;k}_e.
 \end{split}
 \end{displaymath}

 Let us define  
\begin{displaymath}
\Phi_e(d;g,n,k):=\left\langle J^{lead}(\mm^{-})
\ |\ \mm\in \PP(d)\setminus \PP(d,\max\{e-g+1,k-1\})
\right\rangle_\Q.
\end{displaymath}

\begin{prop}\label{prop:compact-type}
The $\Q$-module $\Psi_e(d;g,n,k)/\Phi_e(d;g,n,k)$
 is generated by the classes in
\begin{displaymath}
G_e(d;g,n,k)=\left\{
\big\langle\p\big\rangle_e^{g,n;k}\ \big|\ \p\in\PP(d,e-g+1)\setminus \PP(d,k-1)
\right\}
\end{displaymath}
\end{prop}
\begin{proof}
The coefficients $C_\n^\p$ form a matrix $C$ which is a minor of the 
matrix $M_0(d)$, studied in  Proposition 9.1 from \cite{Rahul-k}. 
It is shown that there is an upper triangular (with respect to the refinement ordering) 
square matrix $Y$ whose rows and columns are indexed by the partitions in $\PP(d)$ 
such that  $M_0(d)Y$ is lower-triangular (with respect to the same order) 
with non-zero diagonal entries. 
The minor $C$ of $M_0(d)$  corresponds to the partitions in 
$$\PP(d)\setminus \PP(d,\max\{e-g+1,k-1\}).$$
The matrix $CY$ is thus lower triangular.
Note that $CY$ is not a square matrix, but the number of its rows is greater than or 
equal to the number of its columns, and it makes sense for $CY$ to be lower triangular with 
non-zero entries on the diagonal.  In particular, the classes $\langle \p \rangle_e^{g,n;k}$
with $\p\in\PP(d)\setminus \PP(d,e-g+1)$ may be expressed in terms of the classes
in $\Phi_e(d;g,n,k)$, as well as the classes in $G_e(d;g,n,k)$. 
\end{proof}

\section{The asymptotic behaviour of the rank of the kappa ring}\label{sec:asymptotic}
Let $f_1,f_2:\Z^{\geq n}\ra \R$ be  real valued functions. If 
$$\limsup_{n\ra \infty}\frac{f_1(n)}{f_2(n)}\leq 0$$
we write $f_1(n)\in \ofrak(f_2(n))$.  
\begin{thm}\label{thm:asymptotic}
Fix the genus $g$, the integer $k\geq 0$, the cycle dimension $e$ and the difference 
$n-d\in\Z$. Then
\begin{displaymath}
\rank\left(\Psi_e(d;g,n,k)\right)-\frac{{n+e\choose e}{g+e\choose e}{k+e\choose e}}{(e+1)!}
\in\ofrak(n^e).
\end{displaymath}
\end{thm}
\begin{proof}
First of all,  for every $\n=\mm^{-}$ as in Section~\ref{sec:localization}, 
using (\ref{eq:reduction}) we have 
\begin{displaymath}
\begin{split}
J^{lead}(\n)&=-\sum_{\Gamma}\sum_{c\in c(\Gamma)}J(\n,\Gamma,c)\\
&=\frac{-1}{m!}\sum_{\substack{\sig\in S_m\\ 0\leq i\leq e\\ \Gamma,c}}
\jmath^\Gamma_*\left(
q_0^*\Big(\eta_i(\sig(\n),\Gamma,c)\Big)q_1^*\left(
\prod_{v\in V_0(\Gamma)}\frac{1}{1-p_v \psi_v}\right)_{e-i}
\cap \left[\Nbar^\Gamma\right]^{vir}\right)\\
&\in \bigoplu_{\Gamma}\bigoplu_{i=0}^{e}
\jmath^{\Gamma}_*\left(R_i(\Nbar^{0,\Gamma})\otimes \Psi_{e-i}
(d;g_\infty(\Gamma),n,k(\Gamma))
\right).
\end{split}
\end{displaymath}
Here, the summation notation is used for summing the subspaces of the tautological 
ring of $\Mbar_{g,n+k}$ and the sums are over all graphs $\Gamma$ 
with $g_\infty(\Gamma)<g$. \\

Proposition~\ref{prop:compact-type} implies that 
$$\rank\left(\Psi_e(d;g,n,k)\right)-\rank\left(\Phi_e(d;g,n,k)\right)\in \ofrak(n^e).$$
We now use induction on the genus $g$ to prove Theorem~\ref{thm:asymptotic}.
The induction hypothesis implies that the rank of 
$$R_i\left(\Mbar^{0,\Gamma}\right)\otimes \Psi_{e-i}(d;g_\infty(\Gamma),n,k(\Gamma))$$
belongs to $\ofrak(n^e)$ unless $i=0$, since the rank of the ring
$R_i(\Mbar^{0,\Gamma})$ does not grow with $n$. In other words, 
\begin{equation}\label{eq:rank-induction}
\rank\left(\frac{\Psi_e(d;g,n,k)}{ \Psi_e(d;g,n,k)\cap \bigoplu_{\Gamma}
\jmath^\Gamma_*\left( \Psi_e(d;g_\infty(\Gamma),n,k(\Gamma))
\right) }\right)\in \ofrak(n^e),
\end{equation}
where we abuse the notation by setting 
\begin{displaymath}
\jmath^\Gamma_*\left( \Psi_e(d;g_\infty(\Gamma),n,k(\Gamma))\right):=
\jmath^\Gamma_*\left( R_0(\Nbar^{0,\Gamma})\otimes
 \Psi_e(d;g_\infty(\Gamma),n,k(\Gamma))\right).
\end{displaymath}
In order to compute the asymptotic  rank of $\Psi_e(d;g,n,k)$, we thus 
restrict ourselves to its subspace which consists of the push-forwards from the strata 
corresponding to the comb graphs $\Gamma$, and over the corresponding cycle
$\Mbar^\Gamma=\Mbar^{0,\Gamma}\times \Mbar^{\infty,\Gamma}$ we  assume 
that the tautological class is the product of the point class from $\Mbar^{0,\Gamma}$ and 
a class in $\Psi_e(d;g_\infty(\Gamma),n,k(\Gamma))$.\\

\begin{figure}
\begin{center}
\includegraphics[scale=.60]{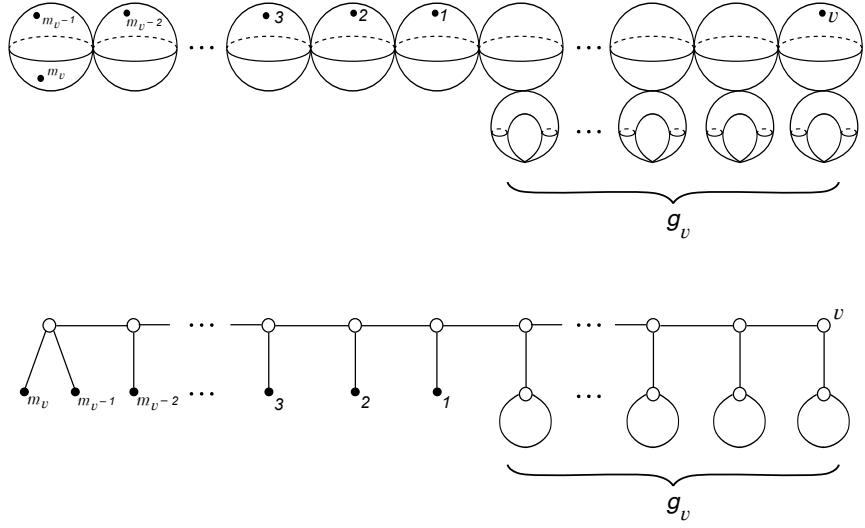} 
 \caption{\label{fig:nodal-curve}\
{The curve $C_v$ of genus $g_v$ with 
$m_v+1$ marked points, and the decorated graph representing it. 
The total number of vertices is $2g_v+2m_v-1$.
The marking $v$ is placed on the special vertex on the upper right corner.}}
\end{center}
\end{figure}

We represent the point class corresponding to the factor $\Mbar^{v,\Gamma}$ of
$\Mbar^{0,\Gamma}$ by a nodal pointed curve $C_v$ which is obtained as follows.
The curve $C_v$ corresponds to the weighted graph illustrated in 
Figure~\ref{fig:nodal-curve}. Applying the above inductive scheme to the 
subspaces $\Psi_{e}(d;g_\infty(\Gamma),n,k(\Gamma))$ we may reduce the genus $g_\infty$ 
to zero. \\

Let $\Gcal$ denote the set of all \SCC with a distinguished vertex 
$v_\infty$ such that $\theta(v_\infty)=(0,n+k_\infty)$ for some $0\leq k_\infty
=k_\infty(G)\leq k$, and for all $v\in V(G)\setminus\{v_\infty\}$, 
$\theta(v)=(0,k_v)$ with $k_v+d_v=3$ (where $d_v$ denotes the degree of 
the vertex $v$). Moreover, we assume that $k_\infty+\sum_v k_v=k$, that 
$G$ has $g$ self-edges and that by deleting these self-edges $G$ becomes a connected 
tree. For $G\in \Gcal$ we get an embedding
$$\imath^G:\Ccal_{G}\simeq \Mbar_{0,k_\infty+d_\infty}\ra \Mbar_{g,n+k}.$$
An inductive use of (\ref{eq:rank-induction}) implies that
\begin{equation}\label{eq:rank-result-of-induction}
\rank\left(\frac{\Psi_e(d;g,n,k)}{\Psi_e(d;g,n,k)\cap \bigoplu_{G\in \Gcal}
\imath^G_*\left( \Psi_e(d;0,n,k_\infty(G))\right) }\right)\in \ofrak(n^e).
\end{equation}
Finally, note that all embeddings $\imath^G$ factor through the embedding 
$$\imath^{g,k}:\Ccal_{g,k}=\Ccal_{G_{g,k}}\ra \Mbar_{g,n+k}$$
which corresponds to the \SCC $G_{g,k}$ with  vertices $v_\infty,1,2,...,g$, such that
  for every $i=1,...,g$ 
  $G_{g,k}$ contains an edge connecting the vertex $i$ to $v_\infty$ together with a
 self edge from $i$ to itself and $\theta(i)=(0,0)$. Moreover,
 $\theta(v_\infty)=(0,k)$. As a result of this observation, from 
 (\ref{eq:rank-result-of-induction}) we obtain
 \begin{equation}\label{eq:asymptotic-reduction}
 \rank\left(\Psi_e(d;g,n,k)/\left(\Psi_e(d;g,n,k)\cap
\imath^{g,k}_*\left( R_e(\Ccal_{g,k})\right) \right)\right)\in \ofrak(n^e).
 \end{equation}
 It is thus enough to prove that 
 \begin{displaymath}
  \rank\left(\Psi_e(d;g,n,k)\cap
\imath^{g,k}_*\left( R_e(\Ccal_{g,k})\right) \right)-\frac{{n+e\choose e}{g+k+e\choose e}}
{(e+1)!}\in \ofrak(n^e)
 \end{displaymath}
 The tautological ring of $\Ccal_{g,k}\simeq \Mbar_{0,n+g+k}$ is generated by 
 combinatorial cycles. 
Let $S_{n,g,k}$ denote the set of permutations in $S_{n+g+k}$ which preserve the 
sets $\{1,...,n\},\{n+1,...,n+g\}$ and $\{n+g+1,...,n+g+k\}$. If $D\subset 
\Mbar_{0,g+n+k}$ is a combinatorial cycle (i.e. one of the boundary strata in 
$\Mbar_{0,g+n+k}$), every  $\sig\in S_{n,g,k}$ acts on $D$ by permuting the markings 
on the curves in $D$ to give a corresponding combinatorial cycle $\sig(D)$.\\

 Suppose that the push-forward $\beta=\imath_*^{g,k}(\alpha)$ belongs to 
 $\Psi_e(d;g,n,k)$ for $\alpha=\sum_{i=1}^N D_i$ with $D_i\in R_e(\Ccal_{g,k})$.
 Since $\Ccal_{g,k}\simeq \Mbar_{0,n+g+k}$, we may set 
 $$\sig(\alpha)=\sum_{i=1}^N \sig(D_i)\ \ \ \forall\ \sig\in S_{n+g+k}.$$
Note that
\begin{displaymath}
\begin{split}
&\beta=\imath_*^{g,k}(\alpha)=\sig(\beta)=\imath_*^{g,k}(\sig(\alpha))\ \ \ \ 
\forall\ \ \sig\in S_{n,g,k}\\
\Rightarrow\ \ &\beta=\imath_*^{g,k}\left(\frac{1}{n!\times g!\times k!}
\sum_{\sig\in S_{n,g,k}}\sig(\alpha)\right)=:\imath_*^{g,k}(\overline{\alpha}).
\end{split}
\end{displaymath} 

The class $\ovl{\alpha}\in R_e(\Mbar_{0,n+g+k})$ is determined by its integrals over 
combinatorial cycles. The $\Q$-module generated by the combinatorial cycles in 
codimension $e$ is the same as $R^e=R_{n+g+k-3-e}(\Mbar_{0,n+g+k})$.
If $D$ is a combinatorial cycle of codimension $e$ in $R^e$ we have
$$\langle \ovl{\alpha}, \sig(D)\rangle=\langle \sig^{-1}(\ovl\alpha), D\rangle=\langle
\ovl\alpha, D\rangle\ \ \ \ \forall\ \ \sig\in S_{n,g,k}.$$
Integration against $\ovl\alpha$ thus gives a map
$$\int_{\ovl\alpha}:\frac{R_{n+g+k-3-e}(\Mbar_{0,n+g+k})}{S_{n,g,k}}\lra \Q$$
which determines $\ovl\alpha$. \\

Every combinatorial cycle $D$ as above determines a combinatorial cycle in 
$R_{n+3g+k-3-e}(\Mbar_{g,n+k})$ as follows. Suppose that $D$ is associated with 
a \SCC $G$ and that 
$$\epsilon_G: V(G)\ra \Z^{\geq 0}\times 2^{\{1,...,n+g+k\}}$$ 
is the corresponding weight function. For every vertex $v\in V(G)$, $\epsilon_G(v)=(0,I_v)$ 
with $I_v$ disjoint subsets of $\{1,...,n+g+k\}$ which give a partition of it.
Let $\pi( G)$ denote the \SCC with the same underlying graph $G$ and the weight function
defined by
\begin{displaymath}
\epsilon_{\pi(G)}(v):=\left(\left|I_v\cap \{n+1,...,n+g\}\right|,I_v\setminus
 \{n+1,...,n+g\}\right)\ \ \ \forall\ v\in V(G).
\end{displaymath}
Let $\pi(D)$ denote the combinatorial cycle associated with $\pi(G)$. If 
$\pi(D)=\pi(D')$ then $\imath_*^{g,k}(D)=\imath_*^{g,k}(D')$.
Moreover, the intersection of $\pi(D)$ with $\Ccal_{g,k}$ is transverse and
\begin{displaymath}
\pi(D)\cap \imath_*^{g,k}\left(\Mbar_{0,n+g+k}\right)=
\#\{D'\ |\ \pi(D')=\pi(D)\}\imath_*^{g,k}(D).
\end{displaymath}
From here we obtain
\begin{displaymath}
\langle \imath_*^{g,k}(\ovl\alpha),\pi(D)\rangle=
\#\left\{D'\ |\ \pi(D')=\pi(D)\right\}\langle \ovl\alpha, D\rangle.
\end{displaymath}
In other words, the map $\int_{\ovl\alpha}$ is determined by the evaluation
\begin{displaymath}
\begin{split}
&\int_\beta:\frac{R_{n+g+k-3-e}(\Mbar_{0,n+g+k})}{S_{n,g,k}}\lra \Q\\
&\int_\beta(D)=\frac{1}{\#\{D'\ |\ \pi(D')=\pi(D)\}}\langle \beta,\pi(D)\rangle.
\end{split}
\end{displaymath}
Since $\beta\in \Psi_e(d;g,n,k)$ the evaluation $\int_\beta(D)$ only depends on the 
modified weight function associated with $D$, and not the underlying graph $G$.
In other words, in order to determine $\int_\beta(D)$ one needs to specify 
\begin{itemize}
\item The dimensions $(d_0,d_1,...,d_e)$ of each one of the $e+1$ components 
of $D$, with the property that $\sum_{i}d_i=n+g+k-3-e$.
\item The number of markings from $\{n+1,...,n+g\}$ on each one of the 
$e+1$ components of $D$.
\item The number of markings from $\{n+g+1,...,n+g+k\}$ on each one of 
the $e+1$ components of $D$.
\end{itemize}
Asymptotically, the number of ways this can be done is equal to 
$$\frac{{n+e\choose e}{g+e\choose e}{k+e\choose e}}{(e+1)!}.$$
This completes the proof of the theorem.
\end{proof}
\begin{cor}\label{cor:asymptotic-kappa}
Fix the codimension $e$ and the genus $g$ and let the number $n$ of the markings grow large. 
Then the rank of $\kappa_e(\Mgnbar)$ as a module over $\Q$ is asymptotic to 
\begin{displaymath}
\frac{{n+e\choose e}{g+e\choose e}}{(e+1)!}.
\end{displaymath}
\end{cor}
\begin{proof}
By Theorem~\ref{thm:asymptotic} we know that
\begin{displaymath}
\rank\left(\kappa_e(\Mgnbar)\right)-\frac{{n+e\choose e}{g+e\choose e}}{(e+1)!}
\in\ofrak(n^e).
\end{displaymath}
Theorem 5 from \cite{ES-k1} implies that 
\begin{displaymath}
\frac{{n+e\choose e}{g+e\choose e}}{(e+1)!}-\rank\left(\kappa_e(\Mgnbar)\right)
\in\ofrak(n^e).
\end{displaymath}
These two observations complete the proof of the corollary.
\end{proof}

\section{The kappa ring of $\Modbar_{1,n}$}
\subsection{The kappa ring and the combinatorial kappa ring}
We would now like to focus on the study of $\kappa^d(\Mbar_{1,n})=\Psi_{n-d}(d;1,n)$.
As before, let
\begin{displaymath}
\Phi_{n-d}(d;1,n)=\left\langle
J^{lead}(\mm^-)\ |\ \mm\in \PP(d)\setminus\PP(d,n-d)\right\rangle.
\end{displaymath}
Other than the comb graphs in $\Jcal=\Jcal(1,n,m,d)$, the only possible comb graphs
are the comb graphs $\Gamma$ with a distinguished vertex $v_0\in V_0(\Gamma)$
with associated genus $g_0=1$, and with $g_v=0$ for all other vertices of $\Gamma$.
The image of the corresponding components of the fixed locus under the forgetful map 
$$\epsilon:\Mbar_{1,n+m}(\Pp(V),d)\ra \Mbar_{1,n}$$
coincides with the image of 
$$\imath=\imath^{1,0}:\Mbar_{0,n+1}\simeq [\mathrm{pt}]\times \Mbar_{0,n+1}
\subset \Mbar_{1,1}\times \Mbar_{0,n+1}
\lra \Mbar_{1,n}.$$
In particular, for every $\mm\in \PP(d)\setminus \PP(d,n-d)$ we have
\begin{displaymath}
J^{lead}(\mm^-)=\imath_*(\kappa(\mm))\ \ \ \ \kappa(\mm)\in\Psi_{n-d}(d;0,n,1).
\end{displaymath}

\begin{thm}\label{thm:relations}
The quotient map from $\kappa^*\left(\Modbar_{1,n}\right)$ to 
$\kappa^*_c\left(\Modbar_{1,n}\right)$ is an isomorphism.
\end{thm}
\begin{proof} 
It is enough to show that if a kappa class $\kappa \in\kappa_e(\Mgnbar)$ is combinatorially
trivial then it is zero. Suppose that $\kappa=\langle \phi\rangle_{1,n}$ for some 
$\phi\in\Psi(d)$ (we refer to Section~\ref{sec:background} for the definitions).
There is a homomorphism 
$$\jmath:\Mbar_{0,n+2}\lra \Mbar_{1,n},$$
which gives an embedding of $\Mbar_{0,n+2}/(\Z/2\Z)$ into $\Mbar_{1,n}$. 
The integral of $\kappa$ over all combinatorial
 cycles of the form $\jmath_*(D)$  is trivial,  since $\kappa$ is combinatorially trivial.
 However, this implies that
 $$\int_{D}\langle \phi\rangle_{0,n+2}=\int_{\jmath_*(D)}\langle \phi\rangle_{1,n}=0
 \ \ \ \ \ \forall\ \ D,$$
 i.e. $\langle \phi\rangle_{0,n+2}=0$.
 By Remark~\ref{remark:rahuls-result}
 $$\langle \phi\rangle_{0,n+2}=0\Rightarrow 
 \phi=\sum_{\mm\in \PP(d)\setminus \PP(d,n-d)} a_\mm
 \left(\sum_{\p\in\PP(d)}C_{\mm^-}^\p.\p\right),$$ 
for some  rational coefficients $a_\mm$, $\mm\in\PP(d)\setminus \PP(d,n-d)$,
 i.e. $\kappa$ is a linear combination of the kappa classes
 $J^{lead}(\mm^-)$. Thus, there is a tautological class $\psi\in R_{d-2}(\Mbar_{0,n+1})$
 such that $\kappa=\imath_*(\psi)$. In particular
 \begin{displaymath}
 \begin{split}
&\langle \psi, D\rangle=\frac{1}{\#\{D'\ |\ \pi(D)=\pi(D')\}}
\langle \kappa, \pi(D)\rangle=0
\ \ \ \forall \ D,\\
\Rightarrow\ \ &\psi=0\ \ \Rightarrow \ \ \kappa=0.
\end{split} 
 \end{displaymath}
 This completes the proof.
 \end{proof}

\subsection{The $\kappa$-trivial combinatorial cycles}
We call a formal linear combination 
$$a_1\q_1+a_2\q_2+...+a_k\q_k\ \ \ a_i\in\Q,\ \ \q_i\in\QQ(d;g,n)$$
a $\kappa$-trivial cycle if for every kappa class $\kappa\in\kappa^d(\Mgnbar)$ 
$$a_1\langle \kappa, \q_1\rangle +a_2\langle \kappa, \q_2\rangle+...+
a_k\langle \kappa, \q_k\rangle=0.$$
The space of $\kappa$-trivial cycles is a subspace of the vector space
$\langle\QQ(d;g,n)\rangle_\Q$ freely generated by the elements of $\QQ(d;g,n)$,
and we denote its rank by $r(d;g,n)$. 
The quotient $V(d;g,n)$ of $\langle \QQ(d;g,n)\rangle_\Q$ by the space of $\kappa$-trivial
cycles is a vector space isomorphic to $\kappa_c^d(\Mgnbar)$. 
Thus, the rank of the combinatorial kappa ring $\kring$ in degree $d$ 
may be computed as 
$$\rank\left(\kappa_c^d(\Mgnbar)\right)=\left |\QQ(d;g,n)\right |-r(d;g,n).$$

\begin{prop}\label{prop:relations}
The rank of $\kappa_c^d(\Mbar_{1,n})$ is at most
$|\PP_1(d,n-d)|$.
\end{prop}
\begin{proof}
Theorem 3 from \cite{ES-k1} 
 implies that  $\frac{1}{24}\q_0(n)-\sum_{i=1}^{n-1}{n-2\choose i-1}\q_i(n-i)$ is  
$\kappa$-trivial. Consequently, for every partition $\n=(n_1,...,n_k)\in\PP(d)$
\begin{equation}\label{eq:genus-1-relation}
\frac{1}{24}\q_0(n,n_1,...,n_k)-\sum_{i=1}^{n-1}{n-2\choose i-1}\q_i(n-i,n_1,...,n_k)
\end{equation}
is $\kappa$-trivial in 
$\left\langle \QQ\left(n+d-k-1;1,n+d\right)\right\rangle_\Q$.
Thus, for every $\n\neq (1,1,...,1)$ in $\PP(n)$,
$\q_0(\n)\in V(n-|\n|;1,n)$ is equal to a linear combination of the cycles $\q_l(\mm)$ for 
$l\geq 1$ and $\mm\in\PP(n-l;|\n|)$. 
In other words, $V(d;1,n)$ is generated by $\q_l(\n)$ for $l\geq 1$ and 
$\n\in\PP(n-l;n-d)$.\\

For $\n=(a,b)\in\PP(n)$ with $a>b$, Theorem 3 from \cite{ES-k1} gives the following 
two equations in $V(n-2;1,n)$:
\begin{equation}\label{eq:g1-simpification}
\begin{split}
\frac{1}{24} \q_0(\n)&=\sum_{i=1}^{a-1}{a-2\choose i-1}\q_i(a-i,b)\\
&=\sum_{i=1}^{b-1}{b-2\choose i-1}\q_i(a,b-i).
\end{split}
\end{equation}
Thus, 
$$\q_{b-1}(a,1,n_1,...,n_k)\in V\left(a+b-2-k+\sum_in_i;1,a+b+\sum_in_i\right)$$  
may be expressed as a linear combination of 
\begin{displaymath}
\begin{split}
&\q_i(a-i,b,n_1,...,n_k),\ \   i=1,...,a-1\ \ \text{and}\\ 
&\q_{j}(a,b-j,n_1,...,n_k),\ 
j=1,...,b-2.
\end{split}
\end{displaymath}
This observation implies that $V(d;1,n)$ is generated by the following elements of
$\QQ(d;1,n)$ (with $l\geq 1$):
\begin{itemize}
\item $\q_l(n_1\leq ...\leq n_{n-d})$ with 
$\sum_{i=1}^{n-d}n_i=n-l$, and $n_1\geq 2$
\item
$\q_l(1\leq n_2\leq ...\leq n_{n-d})$ with $\sum_{i=2}^{n-d}n_i=n-l-1$
and $n_{n-d}\leq l+1$.
\end{itemize} 
Denote the above two sets of generators by $A_1(d;1,n)$ and $A_2(d;1,n)$ respectively, 
and set
 $$A(d;1,n)=A_1(d;1,n)\cup A_2(d;1,n).$$ 

Every element of $A_1(d;1,n)$  corresponds to 
the partition $$(n_1-1,...,n_{n-d}-1)\in \PP(d-l;n-d).$$ 
The size $|A_1(d;1,n)|$  is thus equal to 
$\sum_{l=1}^{2d-n}|\PP(d-l;n-d)|$.
Every partition in $A_2(d;1,n)$  gives the partition 
$$((n_2-1)\leq (n_3-1)\leq ...\leq (n_{n-d}-1)\leq l)\in \PP(d,n-d).$$
Thus, $V(d;1,n)$ is generated by a set of size
\begin{displaymath}
|\PP(d,n-d)|+\sum_{l=1}^{2d-n}|\PP(d-l;n-d)|.
\end{displaymath}
Sending the partition $\n=(n_1\leq ...\leq n_{n-d})\in\PP(d-l;n-d)$ 
to $$\n[l]:=(n_1,...,n_{n-d},1,...,1)\in\PP(d)$$ gives a bijection 
(extending the inclusion $\PP(d,n-d)\subset \PP_1(d,n-d)$)
$$\PP(d,n-d)\cup \coprod_{l=1}^{2d-n}\PP(d-l;n-d)\lra \PP_1(d,n-d).$$
 This completes the proof of Proposition~\ref{prop:relations}.
\end{proof}

\subsection{Independence of the generators.}

\begin{thm}\label{thm-combinatorial-kappa}
The rank of $\kappa_c^d(\Mbar_{1,n})$ is equal to $|\PP_1(d,n-d)|$.
\end{thm}
\begin{proof}
By Proposition~\ref{prop:relations}, it is enough to show that 
the elements of $A(d;1,n)$ are linearly independent in $V(d;1,n)$.\\

If $\n\in\PP(d,n-d)$,  the integral of $\psi(\n)\in\kappa^d(\Mbar_{1,n})$
against every $\q\in A_1(d;1,n)$  is zero, since the length of  $\ppp(\q)$ is $n-d+1$, 
while the length of $\n$ is at most $n-d$ (thus $\n$ does not refine $\ppp(\q)$). 
Meanwhile, the map $\ppp:\QQ(d;1,n)\ra \PP(d,n-d+1)$ gives  an injection
$$\ppp:A_2(d;1,n)\ra \PP(d,n-d).$$ 
 With respect to the refinement ordering on $\PP(d,n-d)$ the matrix 
$$\Big(\left\langle \psi\big(\ppp(\q)\big),\q'\right\rangle\Big)_{\q,\q'\in A_2(d;1,n)}$$
is  triangular with non-zero diagonal entries, and is thus full-rank. 
The above two observations reduce the 
proof of Theorem~\ref{thm-combinatorial-kappa} to showing that  the elements of 
$A_1(d;1,n)$ are linearly independent in $V(d;1,n)$.\\

For every $\n\in \PP(n-l;n-d)$, every $\p\in\PP(d)$, and every integer $N\geq 0$
\begin{displaymath}
\Big\langle \psi(\p),\q_l(\n)\Big\rangle_{1,n}
=\Big\langle \psi(\p),\q_l(\n[N])\Big\rangle_{1,n+N}.
\end{displaymath}

In order to prove the independence of the elements of $A_1(d;1,n)$, it is thus enough to 
prove the independence of the elements of $$A_1^N(d;1,n)\subset A_1(d;1,n+N)$$ 
consisting of $\q_l(\n[N])$ with $\q_l(\n)\in A_1(d;1,n)$.\\

For $\n=(n_1,...,n_k)\in\PP(n)$ and $\mm=(m_1,...,m_p)\in\PP(m)$ define
\begin{displaymath}
\begin{split}
&\bullet\ \ov\n:=(n_1+1,...,n_k+1,1,...,1)\in \PP(2n;n)\ \ \text{and}
\\
&\bullet\ \nn\cup\mm:=(n_1,...,n_k,m_1,...,m_p)\in\PP(m+n)
\end{split}
\end{displaymath}
Let $\tP(d)$ 
denote the set of all sequences $\n=(n_1,...,n_k)$ of positive integers such that 
$n_1+...+n_k=d$. There is a natural map from $\tP(d)$ to $\PP(d)$, which is 
implicitly used below to think of the elements of $\tP(d)$ as partitions.

\begin{lem}\label{lem:k-trivial-2}
For every positive integer $l$ and every $\n\in\PP(n)$ the  cycle
\begin{equation}\label{eq:reduction-to-genus-zero}
\q_l(\n[l])+\frac{1}{24}
\sum_{\mm=(m_1,...,m_k)\in\tP(l)}\left(
\frac{(-1)^{|\mm|} m_1}{d(\mm)}{d(\mm)\choose \mm} 
\right)\q_0(\ov\mm\cup \n)
\end{equation}
 is $\kappa$-trivial.
\end{lem}
\begin{proof}
We use induction on $l$. For $l=1$, Lemma~\ref{lem:k-trivial-2} follows directly from 
(\ref{eq:genus-1-relation}). Suppose now that the claim is proved for $1,2,...,l-1$. 
Using (\ref{eq:genus-1-relation}), for every $\n\in\PP(n)$ 
we make the following computation in $V(n+l;1,k+l)$:

\begin{displaymath}
\begin{split}
\q_{l}(\n[l])&=\frac{1}{24}\q_0(\{l+1\}\cup\n[l-1])-\sum_{i=1}^{l-1}{l-1\choose i}
\q_{l-i}(\{i+1\}\cup \n[l-1])\\
&=\frac{1}{24}\q_0\left(\ov{\{l\}}\cup\n\right)-\\
&\sum_{\substack{i=1,...,l-1\\\mm=(m_1,...,m_k)\in\tP(l-i)}}
\frac{m_1}{l-i}{l-1\choose i}{l-i\choose \mm}\frac{(-1)^{k}}{24}
\q_0\left(\ov{\left(\{i\}\cup\mm\right)}\cup\n\right)
\\
&=-\frac{1}{24}
\sum_{\mm=(m_1,...,m_k)\in\tP(l)}\left(
\frac{(-1)^{|\mm|} m_1}{d(\mm)}{d(\mm)\choose \mm} 
\right)\q_0(\ov\mm\cup \n)
\end{split}
\end{displaymath}
This completes the proof of Lemma~\ref{lem:k-trivial-2}.
\end{proof}
In particular, every element of $A_1^{2d-n}(d;1,n)\subset 
A(d;1,2d)$ is a linear combination (in $V(d;1,2d)$) 
of the cycles of the form $\q_0(\n)$ with $\n\in\PP(2d;d)$
having at least $n-d+1$ terms greater than or equal to $2$. 
Such $\n$'s are determined by $$\mm=\n^{-}\in\PP(d)-\PP(d,n-d).$$
Define $\qqq(\mm)=\q_0(\n)$.\\
 
Let us denote the matrix expressing the elements of 
$A_1^{2d-n}(d;1,n)$ in terms of $\qqq(\mm)$ with $\mm\in\PP(d)-\PP(d,n-d)$ 
by $M(d;1,n)$. The 
rows of $M(d;1,n)$ are thus indexed by the elements of $\PP(d)-\PP(d,n-d)$ 
and its columns are indexed by the elements of $A_1(d;1,n)$. 
In particular, 
$$ \q_l(\n)\in A_1(d;1,n)\ \ \Rightarrow\ \  \n[l]\in \PP(d)-\PP(d,n-d),$$
and the $(\q_l(\n),\n[l])$ component of $M(d;1,n)$ is equal to $\frac{(-1)^{l-1}(l-1)!}{24}$.
Moreover, if $\mm\in\PP(2d,d)$ corresponds to some non-zero entry of  $M(d;1,n)$
in the column corresponding to $\q_l(\n)$ then $\mm^-$ refines $\n$. 
In other words, the square 
sub-matrix of $M(d;1,n)$ corresponding to the rows indexed by $\n[l]$ with 
$\q_l(\n)\in A_1(d;1,n)$ is  triangular with non-zero elements on the 
diagonal (if we use the refinement ordering on the partitions).
Hence $M(d;1,n)$ is a matrix of full rank equal to $|A_1(d;1,n)|$.\\

In order to finish the proof, it is enough to show that the matrix
$$N(d;1,g)=\left(\langle \psi(\p),\qqq(\p')\rangle \right)_{\p,\p'\in \PP(d)-\PP(d,n-d)}$$
is invertible.
This is true since the matrix is upper triangular with non-zero diagonal
elements with respect to the refinement ordering over the partitions. This completes the 
proof of Theorem~\ref{thm-combinatorial-kappa}.
\end{proof}

 \section{The kappa ring of $\Mbar_{2,n}$}
 Let us now consider the case  $g=2$. Fix a partition 
 $$\n\in \PP(d,2d-2-n-l)\ \ \ \ l>0.$$ 
 If the contribution of $\Gamma$ to $J(\n)$ is non-trivial 
 $\Nbar^{0,\Gamma}$ is either trivial, or one of
 \begin{displaymath}
 \Mbar_{1,1},\ \ \ \Mbar_{2,1}\ \ \ \text{or}\ \ \ 
 \Mbar_{1,1}\times \Mbar_{1,1}.
 \end{displaymath}
 Correspondingly, the map $\jmath^{\Gamma}:\Nbar^\Gamma\ra \Mbar_{2,n}$ is one of the 
 four maps
 \begin{displaymath}
 \begin{split}
& \jmath^{0}=Id:\Mbar_{2,n}\ra \Mbar_{2,n}\ \ \ \ \ \ \ \ \ \ \ \ \ \  \ \ \ \ \ \ 
\jmath^{1}:\Mbar_{1,1}\times \Mbar_{1,1}\times \Mbar_{0,n+2}\ra \Mbar_{2,n}\\
&\jmath^{2,n}=\jmath^{2}:\Mbar_{2,1}\times \Mbar_{0,n+1}\ra \Mbar_{2,n}\ \ \ \ \ 
\jmath^{3}:\Mbar_{1,1}\times\Mbar_{1,n+1}\ra \Mbar_{2,n}.
 \end{split} 
 \end{displaymath}
 The comb graphs which correspond to $\jmath^0$ form the leading term $J^{lead}(\n)$
 as their contribution to $J(\n)$. Since a factor $\lambda_1$ appears over either of the 
 two $\Mbar_{1,1}$ components in the domain of $\jmath^1$, the contribution of the 
 comb graphs which correspond to $\jmath^1$ is a class of the form 
 $\imath^{2,n}_*(\kappa(\n))$ for some tautological class 
  $\kappa(\n)\in R^{d-4}(\Mbar_{0,n+2})$. With a similar reasoning, the comb graphs 
  corresponding to $\jmath^3$ contribute via a class of the form
  \begin{displaymath}
  \jmath^3_*\left(\pi_1^*(\lambda_1)\pi_2^*(\psi(\n))\right),\ \ \ \psi(\n)\in 
  \Psi_{n+3-d}(d;1,n,1),
  \end{displaymath}
  where (abusing the notation) $\pi_i$ denotes the projection map  from
   the domain of either of $\jmath^j$ over the $i$-th product factor, for $j=0,1,2,3$ 
   and $i=1,2,3$. \\
  
   Let $\delta$ denote the divisor
  $$\Mbar_{2,1}\setminus \Mod_{2,1}=\frac{[\Mbar_{1,3}]}{\Z/2\Z}
  +[\Mbar_{1,1}\times\Mbar_{1,2}]  =\delta_0+\delta_1.$$ 
By the argument of Section 8 from \cite{M}  over $\Mbar_{2,1}$ we get
 \begin{displaymath}
 \begin{split}
&\kappa_1=2\lambda_1+\frac{1}{2}\delta_1+\psi_1\ \ \ \ \text{and}\ \  \  
 \lambda_1=\frac{1}{12}(\kappa_1+\delta-\psi_1)\\ 
\Rightarrow \ \ &\lambda_2\kappa_1=\frac{4}{5}\lambda_2\lambda_1+\lambda_2\psi_1
\ \ \ \text{and}\ \ \ 
\lambda_2\delta_1=\frac{20}{3}\lambda_2\lambda_1.
 \end{split}
\end{displaymath}  
  Thus, the Chow factor over the component $\Mbar_{2,1}$ in the domain 
  of $\jmath^2$ is a linear combination of $\lambda_2$, $\lambda_2\psi_1$, $\lambda_2\lambda_1$  and the point class. 
  \begin{displaymath}
  \begin{split}
  \end{split}
  \end{displaymath}
  
  For every tautological class 
  $\psi\in R^*(\Mbar_{0,n+1})$ note that 
  \begin{displaymath}
  \jmath^2_*\left(\pi_1^*[\mathrm{pt}]\pi_2^*(\psi)\right)\in
   \Image\left(\imath^{2,n}_*\right).
  \end{displaymath}
Consequently, the total  contribution corresponding to the comb graphs $\Gamma$ with 
$\jmath^\Gamma=\jmath^2$ and $c\in c(\Gamma)$ corresponding to a multiple 
of the point class over $\Mbar_{2,1}$ is of the form $\imath^{2,n}_*(\kappa'(\n))$ for 
some $\kappa'(\n)\in R^{d-4}(\Mbar_{0,n+2})$.
 We thus obtain relations of the form
 \begin{displaymath}
\begin{split}
 J^{lead}(\n)=&\imath^{2,n}_*\big[\kappa(\n)+\kappa'(\n)\big]
 +\jmath^3_*\big[\pi_1^*(\lambda_1)\pi_2^*(\psi(\n))\big]\\
 &+\jmath^2_*\Big[\pi_1^*(\lambda_2)\pi_2^*(\beta(\n))+
 \pi_1^*(\lambda_2\lambda_1)\pi_2^*(\gamma_1(\n))+
 \pi_1^*(\lambda_2\psi_1)\pi_2^*(\gamma_2(\n))\Big]\\
\text{where}\ \ \ \ \  &\beta(\n)\in R^{d-3}(\Mbar_{0,n+1}) \ \ \ \text{and} \ \ \ 
\gamma_i(\n)\in R^{d-4}(\Mbar_{0,n+1}),\ \ i=1,2.\\
 \end{split}
\end{displaymath}

Let us now assume that $\kappa\in \kappa^d(\Mbar_{2,n})$ is combinatorially trivial.
For every \SCC $H$ with the property that the combinatorial cycle associated with 
$H$ is of dimension $d$ and lives in $\Mbar_{2,n}$ we get $\int_{[H]}\kappa=0$.
Applying the above assumption to the stable weighted graphs with the zero genus 
associated with all vertices we find that
$\kappa$ is a linear combination of the classes $J^{lead}(\n)$
by Remark~\ref{remark:rahuls-result}. The above observation implies that
there are classes 
\begin{displaymath}
\begin{split}
&\alpha\in R^{d-4}(\Mbar_{0,n+2}),\ \ \ \ \ \ \ \ \ 
\gamma_i\in R^{d-4}(\Mbar_{0,n+1})\ \ \ i=1,2\\
&\beta\in R^{d-3}(\Mbar_{0,n+1})\ \ \ \text{and}\ \ \ \psi\in \Psi_{n+3-d}(d;1,n,1)
\end{split}
\end{displaymath}
such that
\begin{displaymath}
\begin{split}
\kappa=&\jmath^3_*\big[\pi_1^*(\lambda_1)\pi_2^*(\psi)\big]
+\lambda_2\jmath^1_*\big[\pi_3^*(\alpha)\big]\\
&\ +\lambda_2 \jmath^2_*\Big[\pi_2^*(\beta)+\pi_1^*(\lambda_1)\pi_2^*(\gamma_1)
+\pi_1^*(\psi_1)\pi_2^*(\gamma_2)\Big]
\end{split}
\end{displaymath}
In getting rid of $\imath^{2,n}_*$ and replacing it with $\jmath^1_*$ we are using 
the fact that
$$\frac{1}{24^2}\imath^{2,n}_*(a)
=\jmath^1_*\big(\pi_1^*(\lambda_1)\pi_2^*(\lambda_1)\pi_3^*(a)\big)\ \ \ \forall\ 
a\in R^*(\Mbar_{0,n+2}).$$
 Moreover, since
 the restriction of $\lambda_2$ to  
\begin{displaymath}
 \Mbar_{1,1}\times\Mbar_{1,1}\times\Mbar_{0,n+2}\subset \Mbar_{2,n}
\end{displaymath} 
gives a factor of  $\lambda_1$ over either of the product factors $\Mbar_{1,1}$,
$$\jmath^1_*(\pi_1^*(\lambda_1)\pi_2^*(\lambda_1)\pi_3^*(a))=
\lambda_2\jmath^1_*(\pi_3^*(a)).$$

Consider a \SCC $G$ which determines a combinatorial cycle 
over $\Mbar_{1,n+1}$ with cycle dimension $d-2$ (i.e. of codimension 
$n+3-d$. Let $v$ denote the vertex of $G$ which carries the {\emph{special marking}}
 $n+1$. As in Section~\ref{sec:asymptotic} let
 $\pi(G)$ denote the stable weighted graph obtained from 
$G$ by removing  the special marking from $v$ and increasing the genus 
$g_v$ by $1$. Note that $\pi(G)$  determines a combinatorial cycle
of dimension $d$ in $\Mbar_{2,n}$. Let us assume that the genus associated with all
vertices of $G$ is zero. Since $\pi(G)$ contains a loop, 
the restriction of $\lambda_2$ to $[\pi(G)]$ is trivial. We thus find
\begin{displaymath}
\begin{split}
0=\int_{[\pi(G)]}\kappa
&=\int_{[\pi(G)]}\jmath^3_*\left(\pi_1^*(\lambda_1)\pi_2^*(\psi)\right)
=\frac{1}{24}\int_{[G]}\psi.
\end{split}
\end{displaymath} 
 In particular, the integral of $\psi\in\Psi_{n+3-d}(d;1,n,1)$ over all combinatorial 
 cycles consisting only of genus zero components is trivial. \\
 
 Note that $\psi$ may 
 be represented as the sum of an element $$\psi'\in\Phi_{n+3-d}(d;1,n,1)$$ and a linear 
 combination of the classes in $G_{n+3-d}(d;1,n,1)$. 
 Every linear combination of the 
 classes in $G_{n+3-d}(d;1,n,1)$ is  a linear combination of the classes 
 of the form $\psi_{n+1}^{p-1}\psi(\p)$ with $\p$ a partition in $\PP(d-2-p,n+2-d)$
 by Proposition~\ref{prop:compact-type}.
 Let $\ov{\PP}(d)$ denote the set of marked partitions $(p;\p)$ of $d$, i.e.
 the set of pairs $(p;\p)$ such that $p\leq d$ is a positive integer and  $\p\in\PP(d-p)$.
 Let $\ov{\PP}(d,k)$ denote the subset of $\ov{\PP}(d)$ which consists of the 
 marked partitions $(p;\p)$ with $|\p|<k$. 
 The above observation implies that associated with every 
 $(p;\p)\in\ov{\PP}(d-2,n+3-d)$ there is a  rational coefficient $A_{(p;\p)}$ such that
 \begin{displaymath}
 \psi=\psi'+\sum_{(p;\p)\in\ov\PP(d-2,n+3-d)}A_{(p;\p)}\psi_{n+1}^{p-1}\psi(\p).
 \end{displaymath}
 Since the integral of both $\psi$ and $\psi'$ over all combinatorial cycles which only consist 
 of genus zero components is zero, we conclude that for every such combinatorial 
 cycle $D$ we have 
 \begin{equation}\label{eq:int-against-comb-cycles}
 \sum_{(p;\p)\in\ov\PP(d-2,n+3-d)}A_{(p;\p)}\int_{D}\psi_{n+1}^{p-1}\psi(\p)=0.
 \end{equation}

For a combinatorial cycle $D$ as above  which consists only of genus zero components 
let $p_D-1$ denote the dimension of the component containing the $(n+1)$-th
marking. Let $\p_D$ denote the partition of $d-2-p_D$ determined by the dimensions 
of the rest of the components. Note that the marked partition $(p_D,\p_D)$ of $D$ 
consists of at most $n-(d-3)$ components.   
The integral of $\psi_{n+1}^{p-1}\psi(\p)$ against the combinatorial cycle $D$ only 
depends on $$(p_D;\p_D)\in \PP(d-2,n+3-d).$$ 
The $|\ov\PP(d-2,n+3-d)|\times 
|\ov\PP(d-2,n+3-d)|$ matrix containing all possible integrals 
$\langle \psi_{n+1}^{p-1}\psi(\p), (p_D;\p_D)\rangle$ is triangular with respect to the 
refinement ordering, the equations of (\ref{eq:int-against-comb-cycles}) implies that
$A_{(p;\p)}=0$ for all $(p;\p)\in\ov\PP(d-2,n+3-d)$. In particular, $\psi=\psi'$ belongs 
to $\Phi_{n+3-d}(d;1,n,1)$.\\  

 Since $\psi\in \Phi_{n+3-d}(d;1,n,1)$ it may be expressed in terms of the
  tautological classes pushed forward using
 $$\imath^1:\Mbar_{1,2}\times \Mbar_{0,n+1}\ra \Mbar_{1,n+1}\ \ \text{and}\ \ 
 \imath^2:\Mbar_{1,1}\times \Mbar_{0,n+2}\ra \Mbar_{1,n+1}.$$
 Repeating the argument which was employed at the beginning of this subsection we may 
 write
 \begin{displaymath}
 \begin{split}
& \psi= \imath^1_*\big[\pi_1^*(\lambda_1)\pi_2^*(\gamma_3)\big]+
\imath^2_*\big[\pi_1^*(\lambda_1)\pi_2^*(\alpha')\big]\\
\Rightarrow\ &
 \jmath^3_*\big[\pi_1^*(\lambda_1)\pi_2^*(\psi)\big]=
 \lambda_2\Big[\jmath^1_*\big(\pi_3^*(\alpha')\big)+
 \jmath^2_*\big(\pi_1^*(\delta)\pi_2^*(\gamma_3)\big)\Big]\\
 \Rightarrow\ &
 \kappa=\lambda_2\Big[\jmath^1_*\big(\pi_3^*(\alpha+\alpha')\big)
 + \jmath^2_*\big(\pi_2^*(\beta)\big)\Big]\\
&\ \ \ \ \ \ \ \ \ 
+ \lambda_2\jmath^2_*\Big[\pi_1^*(\lambda_1)\pi_2^*(\gamma_1)
+\pi_1^*(\psi_1)\pi_2^*(\gamma_2)+\pi_1^*(\delta_1)\pi_2^*(\gamma_3)\Big].
 \end{split}
 \end{displaymath}
 The second equality follows since the restriction of $\lambda_2$ to 
 either of 
\begin{displaymath}
\Mbar_{1,1}\times \Mbar_{1,2}\times \Mbar_{0,n+1}, \Mbar_{1,1}\times\Mbar_{1,1}
\Mbar_{0,n+2}\subset \Mbar_{2,n}
\end{displaymath} 
gives a factor of $1=\lambda_0$ over every product factor $\Mbar_{0,\star}$, and 
a factor of $\lambda_1$ over every product factor $\Mbar_{1,\star}$.
 
 The above considerations imply the following lemma.
 \begin{lem}\label{lem:genus-2-rep}
 If the integral of $\kappa\in\kappa^d(\Mbar_{2,n})$ over all combinatorial cycles 
 in $\Mbar_{2,n}$ with the sum of the genera of the components less than $2$
 is trivial then there are 
 tautological classes $a\in R^{d-4}(\Mbar_{0,n+2}), b\in R^{d-3}(\Mbar_{0,n+1})$
 and $c,c'\in R^{d-4}(\Mbar_{0,n+1})$ such that
 \begin{equation}\label{eq:genus-2-rep}
 \kappa=\imath^{2,n}_*(a)+\lambda_2 \jmath^{2,n}_*\Big(
 \pi_2^*(b)+\pi_1^*(\lambda_1)\pi_2^*(c)+\pi_1^*(\psi_1)\pi_2^*(c')
 \Big).
 \end{equation}
 \end{lem}
 \begin{proof}
 Set $a=\alpha+\alpha'$, $b=\beta$, $c=\gamma_ 1+\frac{20}{3}\gamma_3$
 and $c'=\gamma_2$. 
 \end{proof}
 
 \begin{thm}\label{thm:genus-2}
 The quotient map from $\kappa^d(\Mbar_{2,n})\ra \kappa_c^d(\Mbar_{2,n})$
 is an isomorphism.
 \end{thm}
 \begin{proof}
 Every combinatorially trivial kappa class $\kappa$ has a  representation of the  
 form (\ref{eq:genus-2-rep}). Consider a \SCC $G$ which corresponds to a combinatorial
 cycle in $\Mbar_{0,n+2}$. Treating the last two markings in $\Mbar_{0,n+2}$ as the 
 special markings, $\pi(G)$ may be defined as a \SCC determining a combinatorial
 cycle in $\Mbar_{2,n}$. If the intersection of $\pi(G)$ with the image of $\jmath^{2,n}$
 is non-empty then the markings $p=n+1$ and $q=n+2$ lie over the same vertex
 of $G$. In this latter case, the intersection of $[\pi(G)]$ with the image of $\jmath^{2,n}$
 is transverse, unless the vertex of $G$ containing $p,q$ is a vertex $v$ with $d(v)=1$ and 
 $\epsilon(v)=(0,\{p,q\})$. However, if the intersection of $\jmath^{2,n}$ with $[\pi(G)]$
 is transverse then
\begin{displaymath}
\begin{split}
 \int_{[\pi(G)]}\lambda_2\jmath^{1,n}_*
 &\Big[\pi_2^*(b)+\pi_1^*(\lambda_1)\pi_2^*(c)+\pi_1^*(\psi_1)\pi_2^*(c')\Big]\\
&\ \ \ \ \ \ \ \ \  =\int_{[\Image(\jmath^{2,n})]\cap \pi(G)}
 \lambda_2\Big[\pi_2^*(b)+\pi_1^*(\lambda_1)\pi_2^*(c)+\pi_1^*(\psi_1)\pi_2^*(c')\Big].
\end{split}
\end{displaymath} 

 In this case, the intersection includes a factor $\Mbar_{2,1}$, and since the degree of 
 the  Chow class over this factor is at most $3$ the above integral is trivial.
 We conclude that  unless $G$ contains a vertex $v$ with $\epsilon(v)=(0,\{p,q\})$
 and $d(v)=1$ 
 \begin{displaymath}
 \int_{[G]}a=2\times 24^2\times \#\{D\ |\pi(D)=[\pi(G)]\} \times\int_{{\pi(G)}}\kappa=0.
\end{displaymath}

 Let us now assume that the \SCC $G$ has a vertex $v$ with $d(v)=1$ and 
 $\epsilon(v)=(0,\{p,q\})$. Let $w$ be the vertex of $G$ which 
 is connected to $v$  by an edge $e$. The vertex $v$ corresponds to a product 
 factor $\Mbar_{0,3}$ where the three markings are labelled by $\{p,q,e\}$.
The vertex $w$ corresponds to a factor $\Mbar_{0,k+1}$ where the markings  
 are denoted by $\{e,p_1,...,p_k\}$, and with $e$ denoting the marking 
 which corresponds to the edge $e$. For every subset $A\subset \{p_1,...,p_k,p,q\}=B$
 with $2\leq |A|\leq k$ let $G_A$ denote the \SCC obtained as follows. Delete the 
 edges of $G$ which are adjacent to either of $v$ and $w$, except for $e$, 
 to obtain a sub-graph $H$ of $G$ with $V(H)=V(G)$.  If $e'$ denotes a deleted edge 
 of $G$ which connects some vertex $u$ of $G$ to $w$ then $e'$ corresponds to 
 one of the markings $p_i\in\{p_1,...,p_k\}$. If $p_i\in A$ then add an edge to 
 $H$ which connects $u$ to $v$. Otherwise, add an edge to $H$ which connects
 $u$ to $w$. This gives a graph $G_A$. Let $\epsilon(w)=(0,I_w)$.
 Define the weight function over the vertices of  $G_A$ by 
 \begin{displaymath}
 \epsilon_A(u)=\begin{cases}
 \left(0,A\cap (I_w\cup\{p,q\})\right) \ \ \ &\text{if}\ u=v\\
\left(0, (B\setminus A)\cap (I_w\cup\{p,q\})\right) \ \ \ &\text{if}\ u=w\\
\epsilon (u)\ \ \ &\text{otherwise}
 \end{cases}.
 \end{displaymath}
 In particular, $G_{\{p,q\}}=G$ as stable weighted graphs.\\
 
 Keel's Theorem \cite{Keel} implies that
 \begin{displaymath}
 \sum_{\substack{A:\  p,q\in A\\ p_1,p_2\in B\setminus A}}[G_A]
 =\sum_{\substack{A:\  p,p_1\in A\\
 q,p_2\in B\setminus A}}[G_A]
 \end{displaymath}
 In particular, we obtain
 \begin{displaymath}
 \int_{[G]}a=\sum_{\substack{A:\ p,p_1\in A\\ q,p_2\in B\setminus A}}\int_{[G_A]}a
 -\sum_{\substack{A:\ p,q\in A\\ p_1,p_1\in B\setminus A\\ A\neq \{p,q\}}}\int_{[G_A]}a=0.
 \end{displaymath}
 The above discussion implies that the integral of $a$ over all combinatorial cycles is 
 trivial, and thus $a=0$.\\
 
 On the other hand, if $G$ is a \SCC containing a vertex $v$ with $d(v)=1$ and 
 $\epsilon(v)=(0,\{p.q\})$, the combinatorial cycle $[\pi(G)]$ is included in the 
 image of the map $\jmath^{2,n}$. Every such \SCC $G$ corresponds to another \SCC
$G^*$ obtained from $G$ by removing the vertex $v$ from $G$. If $w$ is the unique
vertex adjacent  to $v$ by the edge $e$, we define
\begin{displaymath}
\epsilon_{G^*}(u)=\begin{cases}
\left(0,I_w\cup\{e\}\right)\ \ \ &\text{if }u=w\\
\epsilon(u)\ \ \ &\text{otherwise} 
\end{cases}.
\end{displaymath}
The \SCC $G^*$ determines a combinatorial cycle in $\Mbar_{0,n+1}$, where $e$ 
corresponds to the last marking. Conversely, every combinatorial cycle in 
$\Mbar_{0,n+1}$ is of the form $G^*$. For every $H=G^*$ we obtain
 \begin{displaymath}
 \begin{split}
 0=\int_{[\pi(G)]}\kappa&=\frac{1}{2\times 24^2}\int_{[G]}a+
 \left(\int_{[\Mbar_{2,1}]}\lambda_2\lambda_1\psi_1\right)\int_{[H]}c
 + \left(\int_{[\Mbar_{2,1}]}\lambda_2\psi_1^2\right)\int_{[H]}c'\\
 &= \frac{1}{2\times 24^2}\int_{[H]}(c+\frac{9}{4}c').
 \end{split}
 \end{displaymath}
 Thus $4c+9c'=0$.
 Next, let $G$ be a \SCC which corresponds to a cycle of dimension $d-4$ in 
 $\Mbar_{0,n+1}$. Let $v$ denote the vertex of $G$ which contains the marking
 $n+1$, and let $\epsilon(v)=(0,I_v)$. Let $\tilde{G}$ denote the \SCC obtained from
 $G$ as follows. We add a vertex $w$ to $G$ and connect it to $v$ by a single edge.
 Then we set
 \begin{displaymath}
\epsilon_{\tilde{G}}(u)=\begin{cases}
(1,\emptyset)\ \ \ &\text{if }u=w\\
(1,I_v\setminus\{n+1\})\ \ \ &\text{if }u=v\\
\epsilon(u)\ \ \ &\text{otherwise} 
\end{cases}.
\end{displaymath}
The intersection of the combinatorial cycle $[\tilde{G}]$ with the image of $\jmath^{2,n}$
is always transverse, and they cut each other in 
$$[\delta_1]\times [G]\subset \Mbar_{2,1}\times \Mbar_{0,n+1}.$$
For every \SCC $G$ as above we thus find
\begin{displaymath}
\begin{split}
0=\int_{[\tilde G]}\kappa&=\left(\int_{[\delta_1]}\lambda_2\lambda_1\right)\int_{[G]}c
+\left(\int_{[\delta_1]}\lambda_2\psi_1\right)\int_{[G]}c'=\frac{1}{24^2}\int_{[G]}c'.
\end{split}
\end{displaymath}
Thus $c=c'=0$ and $\kappa=\lambda_2\jmath^{2,n}_*\big(\pi_2^*(b)\big)$.\\
 
 Finally, let $G$ be a \SCC which corresponds to a cycle of dimension $d-3$ in 
 $\Mbar_{0,n+1}$. Let $v$ denote the vertex of $G$ which contains the marking
 $n+1$, and let $\epsilon(v)=(0,I_v)$. Let $\overline{G}$ denote the \SCC obtained from
 $G$ as follows. The graph $\overline{G}$ is obtained from $G$ by adding a pair of 
 vertices $w_1$ and $w_2$ to $G$, which are connected by the edges $e_1$ and $e_2$
 to $v$. We define the corresponding weight function by
  \begin{displaymath}
\epsilon_{\overline{G}}(u)=\begin{cases}
(1,\emptyset)\ \ \ &\text{if }u=w_i,\ \ i=1,2\\
(1,I_v\setminus\{n+1\})\ \ \ &\text{if }u=v\\
\epsilon(u)\ \ \ &\text{otherwise} 
\end{cases}.
\end{displaymath}
The combinatorial cycle $[\overline{G}]$ cuts the image of $\jmath^{2,n}$ transversely in 
$$\left(\Mbar_{1,1}\times \Mbar_{1,1}\times \Mbar_{0,3}\right)\times [G]\subset
\Mbar_{2,1}\times \Mbar_{0,n+1}.$$
For every \SCC $G$ as above we thus obtain
\begin{displaymath}
\begin{split}
0=\int_{[\overline{G}]}\kappa&=\left(\int_{\Mbar_{1,1}}\lambda_1\right)^2\int_{[G]}b
=\frac{1}{24^2}\int_{[G]}b.
\end{split}
\end{displaymath}
Thus $b=0$ and $\kappa$ is  trivial.
\end{proof} 
 


\begin{thebibliography}{Dillo 83}
\bibitem{AC}E. Arbarello, M. Cornalba, Combinatorial and algebro-geometric cohomology
classes on the moduli space of curves, \emph{J. Alg. Geom.} {\bf{5}} (1996), 705–--749.

\bibitem{ES-k1} E. Eftekhary, I. Setayesh, On the structure of the kappa ring, 
{\emph{preprint, available at}} arXiv:1207.2380


\bibitem{F} C. Faber , A conjectural description of the tautological ring of the moduli 
space of curves, {\emph{ Moduli of curves and abelian varieties}}, 109--129, Aspects 
Math., Vieweg, Braunschweig, 1999.

\bibitem{Rahul-lambdag} C. Faber, R. Pandharipande, Hodge integrals, partition matrices, 
and the $\lambda_g$ conjecture, {\emph{Annals of Math.}} {\bf{157}} (2003), 97--124.

\bibitem{Rahul-localization} T. Graber, R. Pandharipande,
Localization of virtual classes,
{\emph{Invent. Math.}} {\bf{135}} (1999), no. 2, 487-518. 

\bibitem{GP} T. Graber and R. Pandharipande, Construction of nontautological classes on 
moduli space of curves, {\emph{ Michigan Math J.}} {\bf 51} (2003), 93--109.

\bibitem{Keel} S. Keel, Intersection theory of moduli space of stable $n$-pointed 
curves of genus zero, {\emph{Trans. of Amer. Math. Socie.}}{\bf 330} (1992) 545--574. 

\bibitem{M} D. Mumford, Towards an enumerative geometry of the moduli space of 
curves, {\emph{ Arithmetic and Geometry}}, (M. Artin and J. Tate, eds.), Part II, 
Birkh\"{a}user, 1983, 271--328

\bibitem{Pet} D. Petersen, The structure of the tautological ring in genus one,
{\emph{preprint, available at}} arXiv:1205.1586.

\bibitem{Rahul-k}
R. Pandharipande, The $kappa$ ring of the moduli of curves of compact 
type, {\emph{Acta Math.}}, {\bf{208}} (2012), 335--388. 
\end{thebibliography}
\end{document}